\documentclass[11pt]{article}
\usepackage{graphicx}
\graphicspath{{./images}}
\usepackage{amsmath,latexsym,amsfonts,amssymb,mathrsfs}
\usepackage[usenames]{color}
\usepackage{tikz}
\usetikzlibrary{math}

\usepackage{url}

\def\N{\mathbb{N}}

\newcommand{\calB}{\mathcal{B}}
\newcommand{\calL}{\mathcal{L}}
\newcommand{\I}{\mathcal{I}}
\newcommand{\bfL}{\mathbf{L}}
\newcommand{\bfH}{\mathbf{H}}

\newcommand{\lb}{[\hspace{-0.02in}[}
\newcommand{\rb}{]\hspace{-0.02in}]}

\newcommand{\vk}{{\bf K}}

\newcommand{\vf}{{\bf f}}

\newcommand{\vn}{{\bf n}}
\newcommand{\vu}{{\bf u}}
\newcommand{\vv}{{\bf v}}
\newcommand{\vs}{{\bf s}}
\newcommand{\calE}{{\mathcal  E}}
\newcommand{\calF}{{\mathcal  F}}
\newcommand{\bchi}{\boldsymbol \chi}

\newcommand{\bvphi}{\boldsymbol \varphi}
\newcommand{\balpha}{\boldsymbol \alpha}
\newcommand{\bbeta}{\boldsymbol \beta}
\newcommand{\vV}{{\bf V}}
\newcommand{\vw}{{\bf w}}
\newcommand{\bH}{{\boldsymbol{H}}}

\newcommand{\bsx}{{\boldsymbol{x}}}
\newcommand{\bsy}{{\boldsymbol{y}}}
\newcommand{\bsz}{{\boldsymbol{z}}}
\newcommand{\bsr}{{\boldsymbol{r}}}

\newcommand{\vsigma}{\boldsymbol{\sigma}}

\newcommand{\R}{\mathbb{R}}

\newcommand{\calS}{\mathcal{S}}

\numberwithin{equation}{section}

\newcommand{\veps}{{\boldsymbol{\varepsilon}}}
\newcommand{\iprod}[1]{\langle#1 \rangle}
\newtheorem{theorem}{Theorem}[section]

\newtheorem{remark}{Remark}[section]

\newenvironment{proof}{\vspace{0.1cm}\hfill\\%
\noindent\textbf{Proof}}{$\Box$
\vspace{0.2cm}\\}

\title{High-order QMC nonconforming FEMs for nearly incompressible planar stochastic elasticity equations\thanks{School of Mathematics and Statistics, University of New South Wales, Sydney, Australia}\thanks{This work was supported by the Australian Research Council
grant DP220101811.}}
\author{J. Dick, T. Le Gia, W. McLean, K. Mustapha, T. Tran}
\begin{document}
\maketitle
\begin{abstract}
In a recent work (Dick et al, QMC sparse grid Galerkin finite element methods for linear elasticity equations with uncertainties, {\tt{https:arxiv.org/pdf/2310.06187.pdf}}), we considered a linear stochastic elasticity equation with random Lam\'e parameters which are parameterized by a countably infinite number of terms in separate expansions. We  estimated the expected values over the infinite dimensional parametric space of linear functionals ${\mathcal L}$ acting on the continuous solution $\vu$ of the elasticity equation. This was achieved by truncating the expansions of the random parameters, then using a high-order quasi-Monte Carlo (QMC) method to approximate the high dimensional integral combined with the conforming Galerkin finite element method (FEM) to approximate the displacement over the physical domain $\Omega.$ In this work, as a further development of aforementioned article, we focus on the case of a nearly incompressible linear stochastic elasticity equation. To serve this purpose,  in the presence of stochastic inhomogeneous (variable Lam\'e parameters) nearly compressible material,   we develop a new  {\em locking-free} {\em symmetric} nonconforming Galerkin FEM that handles the inhomogeneity. In the case of nearly incompressible material, one known important advantage of nonconforming approximations is that they yield optimal order convergence rates that are uniform in the Poisson coefficient. 
Proving the convergence of the nonconforming FEM leads to another challenge that is summed up in showing the needed regularity properties of $\vu$.  For the error estimates from the high-order QMC method, which is needed to estimate the expected value over the infinite dimensional parametric space of ${\mathcal L}\vu,$ we 
are required here to show certain regularity properties of $\vu$ with respect to  the random coefficients.
Some numerical results are delivered at the end.   
\end{abstract}
\section{Introduction}
In this work we explore and analyze the application of  the quasi-Monte Carlo (QMC)  method combined with a new nonconforming  Galerkin finite element  method (FEM)  to solve a linear elastic model with uncertainties \cite{Eigel2014,HoangNguyenXia2016,Khan2019,Khan2021,XiaHoang2014}. More specifically,  we consider the case  where the properties of the elastic inhomogeneous material (that is, the Lam\'e parameters) are varying spatially in an uncertain way (see \eqref{KLexpansion}). The equation governing small elastic deformations of a body $\Omega$ in~$\R^2$ with polygonal boundary can be written as
\begin{equation}
 -\nabla \cdot \vsigma(\bsr;\vu(\bsx,\bsr)) = \vf(\bsx) \quad \text{for } \bsx \in \Omega,~~\bsr=(\bsy,\bsz), \label{eq:L1}
\end{equation}
subject to homogeneous Dirichlet boundary conditions $\vu(\cdot,\bsr) = {\bf 0}$ on   $\Gamma:=\partial \Omega$ with $\bsy$ and $\bsz$ being parameter vectors describing randomness. The parametric Cauchy stress tensor $\vsigma(\bsr;\cdot) \in [L^2(\Omega)]^{2\times 2}$ is defined as 
\[\vsigma(\bsr;\vu(\cdot,\bsr)) = \lambda(\cdot,\bsz)\Big(\nabla\cdot \vu(\cdot,\bsr)\Big) I 
+ 2\mu(\cdot,\bsy) \veps(\vu(\cdot,\bsr))\quad{\rm on}\quad \Omega \] 
with  $\vu(\cdot, \bsr)$ being the displacement vector field and the symmetric strain tensor
 $ \veps(\vu) := \frac{1}{2} (\nabla \vu + (\nabla \vu)^T).$
Here, $\vf$ is the body force per unit volume, and $I$ is the identity tensor.   The gradient ($\nabla$) and the divergence ($\nabla \cdot$) are understood to be with respect to the physical variable $\bsx \in \Omega$.  The random parameters $\mu$ and $\lambda$ are expressed in the following  separate expansions: 
\begin{equation}\label{KLexpansion}
 \mu(\bsx,\bsy) = \mu_0(\bsx) +\sum_{j=1}^\infty y_j \psi_j(\bsx)~~{\rm and}~~ 
\lambda(\bsx,\bsz) =\Lambda  \widehat \lambda(\bsx,\bsz) = \Lambda\Big(\widehat \lambda_0(\bsx) +\sum_{j=1}^\infty z_j \phi_j(\bsx)\Big),
\end{equation}
where $\bsy= (y_j)_{j\ge1}$ and $\bsz=(z_k)_{k\ge1}$ belong to $U:=(-\frac{1}{2},\frac{1}{2})^\N$ consisting of a countable number of parameters $y_j$ and $z_k$, respectively,  which are assumed to be i.i.d. uniformly distributed. 
Here, the constant~$\Lambda$ satisfies $1\le\Lambda<\infty$, and $\{\psi_j\}$ and $\{\phi_j\}$ are orthogonal basis functions for $L^2(\Omega)$. For the well-posedness of the  elastic problem \eqref{eq:L1}, we assume that
$\mu_{\min} \le \mu \le \mu_{\max}$ and $\widehat \lambda_{\min}\le \widehat \lambda\le \widehat \lambda_{\max}$  on $\Omega \times U,$  for some positive constants $\mu_{\min}$, $\mu_{\max}$, $\widehat \lambda_{\min},$  and $\widehat \lambda_{\max}.$

To ensure that  $\mu$ and $\widehat \lambda$ are well-defined for all
parameters $\bsy$, $\bsz \in U$, we impose  that
\begin{equation}\label{ass A1}
 \mu_0,\,\widehat \lambda_0  \in L^\infty(\Omega), \quad \sum_{j=1}^\infty \|\psi_j\|_{L^\infty(\Omega)} < \infty~~{\rm and}~~
 \sum_{j=1}^\infty \|\phi_j\|_{L^\infty(\Omega)} < \infty\,.
\tag{A1}
\end{equation}

As in \cite{DickEtAl2022}, in this work, we are interested in efficient  approximation of the expected value 
\begin{equation}\label{F} 
\I(\calL(\vu)):=\int_{U}\int_{U}\calL\big( \vu(\cdot,\bsr)\big)\,d\bsr
\end{equation}
of $\calL(\vu(\cdot,\bsr))$ for a certain linear functional $\calL:{\bf L}^2(\Omega)  \to \R$, with respect to the random vector variables $\bsr=(\bsy,\bsz),$ $d\bsr=d\bsy\,d\bsz$,  $d\bsy$ and $d\bsz$ are the uniform probability measures on $U$, and  $\vu(\cdot,\bsr)$ (displacement) is the random vector function that solves problem  \eqref{eq:L1}. 
The main focus is the elastic material that is permitted to be nearly incompressible, that is,  the Poisson ratio of the elastic material $\nu$ approaches $1/2$, which is not the case in \cite{DickEtAl2022}. So, in the current analysis, we seek estimates that are independent of the (possibly very large) constant factor $\Lambda$ in \eqref{KLexpansion}.
Note that, as a special case,  if $\lambda$ is a constant multiple of $\mu$, that is, the randomnesses  of the Lam\'e parameters are due to the ones in the Young modulus $E$, then  $\nu=\frac{\lambda}{2(\lambda+\mu)}$ is constant, see \cite{Khan2019,Khan2021}.


To estimate $\I(\calL(\vu))$,  we have to deal with three sources of errors: a dimension truncation error from truncating the  infinite expansions in \eqref{KLexpansion}, a QMC quadrature error from approximating the high-dimensional integral, and a nonconforming Galerkin  error  from approximating the continuous solution over the physical domain $\Omega$. We explore these steps and state the main results of this work (Theorem \ref{sec: main results}) in the next section.  Prior to this, we discuss the estimation of the  expected value  of a given high dimensional function via a two different high-order QMC rules. The first one is based on a tensor product structure of  two separate high-order rules for the random coefficients arising from the expansions for $\lambda$ and  $\mu$, respectively, see Theorem~\ref{prop:qmc}. The second high-order QMC rule is based on  simulating both $\lambda$ and $\mu$ simultaneously using one family of QMC points. Under some additional regularity assumptions, the latter QMC rule leads to a better rate of convergence, see Theorem \ref{prop:qmc6.3}. We refer to  \cite{DickEtAl2022} for  the proofs of Theorems \ref{prop:qmc} and \ref{prop:qmc6.3}.  

 The achieved regularity properties in  \cite{DickEtAl2022} are not applicable here due to the $\Lambda$ dependence. So, in Section \ref{VFR}, we  show in detail the required regularity properties of the continuous solution $\vu$  (that are $\Lambda$ independent) with respect to both the random parameters $\bsy$ and $\bsz$, and the spatial variable $\bsx.$ These results are needed to guarantee the convergence of the errors from both the QMC integration and the nonconforming finite element discretization. For the convergence of the latter, we require the parameters $\mu(\cdot,\bsy),\,\widehat \lambda(\cdot,\bsz)\in W^{1,\infty}(\Omega)$ for every $\bsy,\,\bsz  \in U$. To have this, we additionally assume that
\begin{equation}\label{ass A4}
\mu_0,\,\widehat \lambda_0  \in W^{1,\infty}(\Omega),~~ \sum_{j=1}^\infty \|\nabla \psi_j\|_{L^\infty(\Omega)} < \infty~~ {\rm and}~~ \sum_{j=1}^\infty \|\nabla \phi_j\|_{L^\infty(\Omega)} < \infty.
\tag{A2}
\end{equation}

The error from truncating the infinite series expansion  in \eqref{KLexpansion} is investigated in Section \ref{Sec: Truncation} (Theorem \ref{Truncating error}). More precisely, if  $\vu_{\vs}$, $\vs:=(s_1,s_2) \in \N^2$, is the solution of \eqref{eq:L1} obtained by truncating the infinite expansions in \eqref{KLexpansion} at $j=s_1$ and $k=s_2$, then the task is to estimate the error between $\vu$ and $\vu_{\vs}$. In other words, 
$\vu_{\vs}$ is the solution of \eqref{eq:L1}  corresponding to the truncated (finite) Lam\'e parameters $\mu_{s_1}(\bsx,\bsy^{s_1}) = \mu(\bsx,\bsy^{s_1})$ and $\lambda_{s_2}(\bsx,\bsz^{s_2}) =\lambda(\bsx,\bsz^{s_2}),$ where    $\bsy^{s_1}=(y_1,y_2,\cdots,y_{s_1},0,0,\cdots)$ and  $\bsz^{s_2}=(z_1,z_2,\cdots,z_{s_2},0,0,\cdots)$.
To  minimize  the errors  (see Theorem \ref{Truncating error}), we  assume that the functions  $\psi_j$ and $\phi_j$  are ordered
so that $\|\psi_j\|_{L^\infty(\Omega)}$ and $\|\phi_j\|_{L^\infty(\Omega)}$ are  nonincreasing, that is, 
\begin{equation}\label{ass A5}
\| \psi_j \|_{L^\infty(\Omega)} \ge
\| \psi_{j+1} \|_{L^\infty(\Omega)} ~~{\rm and}~~ \| \phi_j \|_{L^\infty(\Omega)} \ge
\| \phi_{j+1} \|_{L^\infty(\Omega)},\quad{\rm for}~~j\ge 1.
\tag{A3}
\end{equation}
We may also require that (this assumption is needed to show a high-order QMC error as well)
\begin{equation}\label{ass A3}
 \sum_{j=1}^\infty \| \psi_j \|^p_{L^\infty(\Omega)} < \infty
 ~~{\rm and}~~  \sum_{j=1}^\infty \| \phi_j \|^q_{L^\infty(\Omega)} < \infty,\quad{\rm for~some}~~0<p,\,q<1\,.
\tag{A4}
\end{equation}

To avoid the locking and to deal with the variability of the Lam\'e parameters, we develop a new symmetric low-order (piecewise linear)  nonconforming Galerkin FEM for approximating the parametric solution $\vu(\cdot,\bsr)$ of \eqref{eq:L1} over the physical domain $\Omega$  in Section~\ref{Sec: FEM},  and prove error estimates.    It is known that conforming FEMs (studied in  \cite{DickEtAl2022})  suffer from so called locking (or non-robustness) \cite{AinsworthParker2022} in the nearly incompressible case.  The locking phenomenon was avoided by using different approaches; these include   nonconforming and   mixed Galerkin FEMs  \cite{BottiDi PietroGuglielmana2019,Braess2007, BrennerSung1992,BrennerScott2008, BrezziFortin1991,Falk1991} for the case of constant Lam\'e parameters. The nonconforming FEM in the existing literature is not applicable for the case of variable   Lam\'e parameters. Section \ref{Sec: Numeric} is devoted to numerical simulations. In four different examples, we numerically illustrate the achieved theoretical convergence results.

\begin{remark}
Apart from the regularity results in Theorem \ref{lem: vu bound}, the achieved  results can be extended to  the case of homogeneous mixed boundary conditions of the form $\vu(\cdot,\bsr)=0$ on $\Gamma_D$ and $\vn\cdot\vsigma(\bsr;\vu(\cdot,\bsr))=0$ on $\Gamma_N$, where $\Gamma_D$ and $\Gamma_N$ are disjoint, with  $\Gamma_D \cup \Gamma_N=\Gamma$ and  measure of $\Gamma_D>0$, for every $\bsr \in U\times U$. Here,  $\vn$ is the outward unit normal column vector to $\Gamma_N.$
\end{remark}

\section{Estimations of the expected value}\label{Sec: main results}
To estimate the expected value $\I(\calL(\vu))$ in \eqref{F}, we initially approximate  $\vu$ by $\vu_\vs$ ($\vs:=(s_1,s_2) \in \N^2$) which is the solution of \eqref{eq:L1} corresponding to the truncated $\mu_{s_1}$ and $\lambda_{s_2}$. Then, with    $U_i=[0,1]^{s_i}$ being of fixed dimensions $s_i$ for $i=1,2,$  we intend to approximate $\I(\calL(\vu))$ by  
\begin{equation}\label{F finite}
 \I_\vs(\calL(\vu_{\vs})):=\int_{U_2}\int_{U_1}F(\bsr)d\bsr,\quad{\rm where}~~F(\bsr):=\calL\Big(\bf \vu_\vs\big(\cdot,\bsr-{\bf \frac{1}{2}}\big)\Big)\,.
\end{equation}
The shifting of the coordinates by ${\bf \frac{1}{2}}$ translates $U_i$ to  $\big[-\frac{1}{2},\frac{1}{2}\big]^{s_i}$ for $i=1,2$. The idea behind using this shift is due to the second step where  we  approximate the high dimensional integrals in \eqref{F finite} using \emph{deterministic, interlaced high-order polynomial lattice QMC rules}  of the form  
\begin{equation}\label{eq:QMCInt F}
\I_\vs(\calL(\vu_\vs)) \approx \I_{\vs,{\bf N}}(\calL(\vu_\vs)):=\frac{1}{N_1\,N_2}\sum_{k=0}^{N_2-1}\sum_{j=0}^{N_1-1} F(\bsr_{j,k})~~{\rm with}~~{\bf N}=(N_1,N_2)\,,
\end{equation}
where $\bsr_{j,k}=(\bsy_j,\bsz_k)$, and the QMC points  $\{\bsy_0, \ldots,\bsy_{N_1-1} \} \in U_1$ and $\{\bsz_0, \ldots,\bsz_{N_2-1} \} \in U_2$ (the QMC points are always constructed over the unit cube $[0,1]^s$ making the shift by $1/2$ necessary). With $i=1,2,$ we generate the $N_i: = b^{m_i}$ ($m_i\ge 1$ is integer and $b$ is prime)  QMC points using a \emph{generating vector} of polynomials with degree less than $ m_i$   with its coefficients taken from a finite field ${\mathbb Z}_{b}$; see \cite[Section 6]{DickEtAl2022} for full details.  
\begin{theorem}[{\cite[Theorem 6.1]{DickEtAl2022}}] \label{prop:qmc}
Let  $\bchi=(\chi_j)_{j\ge 1}$ and $\bvphi = (\varphi_j)_{j\ge 1}$ be two sequences of positive numbers such that $\sum_{j=1}^\infty \chi_j^p$ and $\sum_{j=1}^\infty \varphi_j^q$ being finite for some $0<p,\,q<1.$ Let $ \bchi_{s_1}=(\chi_j)_{1\le j\le s_1}$, $\bvphi_{s_2} = (\varphi_j)_{1\le j\le s_2}$, $\alpha \,:=\, \lfloor 1/p \rfloor +1$, and $\beta \,:=\, \lfloor 1/q \rfloor +1$. Assume that there exists a positive constant $c$ such that
\begin{equation} \label{eq:like-norm}
|\partial^{\balpha}_\bsy F(\bsr)| \le c|\balpha|! 
\bchi_{s_1}^{\balpha}\quad{\rm and}\quad |\partial^{\bbeta}_\bsz F(\bsr)| \le c|\bbeta|! 
\bvphi_{s_2}^{\bbeta},
\end{equation}
for all $\bsr=(\bsy,\bsz)\in U_1\times U_2$, 
$\balpha \in \{0, 1, \ldots, \alpha\}^{s_1}$, and $\bbeta \in \{0, 1, \ldots, \beta\}^{s_2}$.
Then one can construct two interlaced  high-order polynomial lattice rules of order $\alpha$ with $N_1=b^{m_1}$ points, and of order $\beta$ with $N_2=b^{m_2}$ points, with $|m_1q - m_2 p| < 1$, using a fast component-by-component (CBC) algorithm,  such that the quadrature  error  $|\I_\vs(\calL(\vu_\vs)) - \I_{\vs,{\bf N}}(\calL(\vu_\vs))|\le C\, N^{-\frac{1}{p+q}}$, where $N = N_1 N_2$ is the total number of QMC quadrature points. The constant $C$ depends on $p,q,$ and $b$, but is independent of $s_1$, $s_2$, $m_1$ and $m_2$. 
\end{theorem}

In the next theorem, we state the error from approximating the  integral  in \eqref{F finite} by a QMC rule in dimension $s_1+s_2$ directly. In this approach we combine the different weights arising from simulating $\mu$ and $\lambda$. The proof follows directly from \cite[Theorem 3.1]{DickKuoGiaNuynsSchwab2014}.
\begin{theorem}\label{prop:qmc6.3}
 Let $\bchi$ and $\bvphi$ be the two sequences introduced in Theorem \ref{prop:qmc}, and let $\gamma \,:=\min(\lfloor 1/p \rfloor,\lfloor 1/q \rfloor) +1.$ For any   $\boldsymbol{\gamma} \in \{0, 1, \ldots, \gamma\}^{s_1+s_2}$, assume that $F$ satisfies
 \begin{equation} \label{eq:like-norm mixed}
|\partial^{\boldsymbol{\gamma}}_{\bsr} {F}(\bsr)|=|\partial^{\boldsymbol{\gamma}_1,\boldsymbol{\gamma}_2}_{\bsy,\bsz} {F}(\bsy,\bsz)| \le c|\boldsymbol{\gamma}|! 
\bchi_{s_1}^{\boldsymbol{\gamma}_1} 
\bvphi_{s_2}^{\boldsymbol{\gamma}_2},~~{\rm for~any}~~\bsr=(\bsy,\bsz)\in U_1\times U_2=[0,1]^{s_1+s_2}, 
\end{equation}
where the vectors $\boldsymbol{\gamma}_1$ and $\boldsymbol{\gamma}_2$ are formed from the first $s_1$ and last $s_2$ components of $\boldsymbol{\gamma},$ respectively, and  the constant $c$ is independent of $\bsr$, $s_1$, $s_2$, and of $p$ and $q.$ Then one can construct an interlaced polynomial lattice rule of order $\gamma$ with $N=b^m$  points using a fast CBC algorithm  so that 
\[|\I_\vs(\calL(\vu_\vs)) - \I_{\vs,{\bf N}}(\calL(\vu_\vs))|=\Big|\I_\vs(\calL(\vu_\vs))-\frac{1}{N} \sum_{n=0}^{N-1} \calL(\vu_\vs(\boldsymbol{r}_n-\frac12))\Big|\le C\,N^{-\min(1/p,1/q)}\,,\]
where the generic constant $C$ depends on $b,p$ and $q$, but is independent of ${\bf s}$ and $m.$ 
\end{theorem}

In the third step of our error analysis, we  approximate  the unknown truncated 
solution~$\vu_\vs$ over the physical domain $\Omega$  by $\vu_{{\vs},h} \in \vV_h$ (see
Section~\ref{Sec: FEM} for the definition of the finite element space~$\vV_h$). However,
before summarising these steps  in the upcoming  theorem,  we introduce  the following 
spaces and notations.  Let $\|\vw\|$ denote the norm of a vector field~$\vw$ 
in~${\bf L}^2(\Omega)=[L^2(\Omega)]^2$, and let $\|\vw\|_{\bf V}$~and $\|\vw\|_{\bf H}$ 
denote the norms of $\vw$ in $\vV=[H^1_0(\Omega)]^2$~and ${\bf H}=[H^2(\Omega)]^2$, 
respectively.
Here, $H^1_0(\Omega)$~and $H^2(\Omega)$ are the usual Sobolev spaces of scalar-valued 
functions on~$\Omega$.  Finally, $\vV^*$ is  the dual space of~$\vV$ with respect to 
the inner product in~${\bf L}^2(\Omega)$.


In the next theorem,  the generic constant  $C$ is independent of the finite element mesh size~$h$,  the dimensions $s_1$ and $s_2,$  the number of QMC point    $N$, and the Lam\'e parameter factor  $\Lambda$. However, $C$ may depend on $\Omega$ and other parameters including  $p,q,b,$ the upper/lower bounds of $\mu$ and $\widehat \lambda$, and the upper bounds of  $\|\nabla \mu\|_{L^\infty(\Omega)}$ and $\|\nabla \widehat \lambda\|_{L^\infty(\Omega)}.$  Write $\bsr^\vs=(\bsy^{s_1},\bsz^{s_2})$ for the truncated version
of~$\bsr=(\bsy,\bsz)$.

\begin{theorem}\label{sec: main results}
Assume that $\Omega$ is convex  and  \eqref{ass A1}--\eqref{ass A3} are satisfied. For every $\bsr=(\bsy,\bsz) \in U \times U,$ let $\vu \in {\bf H}$ be the solution of problem \eqref{eq:L1} and $\vu_{\vs,h}$ be the nonconforming finite element solution, defined as in \eqref{FEM new} with $\bsr^\vs$ in place of $\bsr$ for $s_1,s_2 \ge 1.$  Let $\alpha \,:=\, \lfloor 1/p \rfloor +1$ and $\beta \,:=\, \lfloor 1/q \rfloor +1$, where $p$ and $q$ are those in \eqref{ass A3}. Then one can construct two interlaced   high-order polynomial lattice rules of order $\alpha$ with $N_1$ points, and of order $\beta$ with $N_2$ points,  such that 
\[    |\I(\calL(\vu))-\I_{\vs,{\bf N}}(\calL(\vu_{\vs,h}))|\le C\,\Big(s_1^{1-1/p}+ s_2^{1-1/q}+N^{-1/(p+q)}\Big)   \|\vf\|_{\vV^*}   \|\calL\|_{\vV^*}
+C \, h^2 \|\vf\| \|\calL\|\,.\]

Further, if we use a QMC rule in dimension $s_1 + s_2$ directly with $N$ points as in Theorem~\ref{prop:qmc6.3}, then the above estimate remains valid with  $N^{-\min(1/p, 1/q)}$ in place of $N^{-1/(p+q)}$. 
\end{theorem}
\begin{proof} We start our proof by noticing that the error $|\I(\calL(\vu))-\I_{\vs,{\bf N}}(\calL(\vu_{\vs,h}))|$ is bounded by  
\begin{equation}\label{combine}
 |\I(\calL(\vu))-\I_{\vs}(\calL(\vu_\vs))|
+|\I_{\vs}(\calL(\vu_\vs))-\I_{\vs,{\bf N}}(\calL(\vu_\vs))|+|\I_{\vs,{\bf N}}(\calL(\vu_\vs-\vu_{\vs,h}))|.
\end{equation} 
Since $d\bsy$ and $d\bsz$ are the uniform probability measures with  i.i.d. uniformly distributed parameters on $U$, by using Theorem \ref{Truncating error} we obtain 
\[\I(\calL(\vu))-\I_{\vs}(\calL(\vu_\vs))
=\int_{U}\int_{U}\calL\Big(\vu(\cdot,\bsr)-\vu_{\vs}(\cdot,\bsr^\vs)  \Big)\,d\bsr\le C\,\Big(s_1^{1-1/p}+  s_2^{1-1/q}\Big)   \|\vf\|_{\vV^*} \|\calL\|_{\vV^*}\,.\]
To bound the second term in \eqref{combine}, which is the QMC quadrature error, we apply    Theorem~\ref{prop:qmc} with $F(\bsr):=\calL\Big(\bf \vu_\vs\big(\cdot,\bsr-{\bf \frac{1}{2}}\big)\Big).$ Since $|\partial_{\bsr}^{\balpha,\bbeta}F(\bsr)|\le \|\calL\|_{\vV*} \,\|\partial_{\bsr}^{\balpha,\bbeta}\vu_\vs\big(\cdot,\bsr-{\bf \frac12}\big)\|_{\vV}$, the needed regularity conditions in \eqref{eq:like-norm} are satisfied owing to Theorem \ref{lem: vu bound y z}. Therefore, 
\[|\I_{\vs}(\calL(\vu_\vs))-\I_{\vs,{\bf N}}(\calL(\vu_\vs))|\le C\,N^{-\frac{1}{p+q}}   \|\vf\|_{\vV^*} \|\calL\|_{\vV^*}\,.\]

To estimate the third term in \eqref{combine}, we refer to  Theorem \ref{Convergence theorem QMC} by taking the vectors $\bsy$ and $\bsz$ with $y_j = z_k=0$ for $j>s_1$ and $k>s_2$. Hence, the proof of this theorem is complete. 
\end{proof}

\section{Regularity properties}\label{VFR}
This section is devoted to showing some useful regularity properties of the weak solution with respect to both the  physical variable $\bsx$ and parametric variables $\bsy$ and $\bsz$. Following \cite{DickEtAl2022},  the weak formulation of parameter-dependent equation  \eqref{eq:L1} is:  Find $\vu(\cdot,\bsr) \in \vV$ so that
\begin{equation}\label{para weak}
 \calB(\bsr;\vu(\cdot,\bsr), \vv) = \ell(\vv), \quad \text{for all~~$\vv \in \vV$ and  for every  $\bsr=(\bsy,\bsz) \in U\times U$,}
\end{equation}
 where  the linear functional $\ell$ and the bilinear operator~$\calB$ are defined by
\begin{equation*}\label{eq: bilinear}
 \ell(\vv) = \int_\Omega \vf \cdot \vv \,d\bsx 
 \quad\text{and}\quad
  \calB(\bsr;\vu(\cdot,\bsr), \vv) = \int_\Omega [2\mu\, \veps(\vu(\cdot,\bsr)):\veps(\vv)+\lambda \nabla \cdot \vu(\cdot,\bsr) \nabla \cdot \vv] \,d\bsx.
\end{equation*}
 The colon operator is the inner product between tensors. Since $\max\{\frac{\|\nabla \cdot \vw\|}{\sqrt{2}},\|\veps(\vw)\|\} \le  \|\nabla \vw\|$,  
\begin{equation}\label{eq: bounded}
\calB(\bsr;\vv,\vw) \le C_{\lambda,\mu}\|\nabla \vv\|\,\|\nabla \vw\|\le 
C_{\lambda,\mu}\| \vv\|_{\vV}\,\|\vw\|_{\vV},~~{\rm for}~~\vv,\vw \in \vV.
\end{equation}
 So,  $\calB(\bsr;\cdot,\cdot)$ is bounded over $\vV \times \vV$. The coercivity property  of $\calB(\bsr;\cdot,\cdot)$ on $\vV$  follows from the nonnegativity assumptions on $\mu$ and $\lambda$ and Korn's inequality as  \begin{equation}\label{eq: coer B}
 \calB(\bsr;\vv,\vv)\ge  2 \mu_{\min} \|\veps(\vv)\|^2 \ge C \mu_{\min} \| \vv\|_{\vV}^2, \quad \vv \in \vV.
\end{equation}
Owing to these two properties, and since $|\ell(\vv)|\le \|\vf\|_{\vV^*}\|\vv\|_{\vV}$,  an application of the Lax-Milgram theorem completes the proof of the next theorem.   
\begin{theorem}\label{thm: unique solution}
For every $f \in \vV^*$, and for every $\bsr  \in U\times U$, problem \eqref{para weak} has a unique solution. 
\end{theorem}
\subsection{Regularity in $\bsx$}
For the convergence of the nonconforming finite element solution, we derive next some regularity estimates of the parametric solution of \eqref{para weak}. For the planar  linear elasticity problem with constant Lam\'e parameters $\lambda$ and $\mu$, these estimates were proved in  \cite[Theorem A.1]{Vogelius1983} for domains with smooth boundary. Later on, some ideas were borrowed  from \cite{Vogelius1983} to show these results  for the case of convex polygonal domains, see \cite[Lemma 2.2]{BrennerSung1992}. Our results extend the ones in the existing literature  to the case of variable  random parameters $\lambda$ and $\mu$. Some ideas from \cite[Lemma 2.2]{BrennerSung1992} are used in our proof. 
\begin{theorem}\label{lem: vu bound}
Assume that \eqref{ass A4} is satisfied. Then, for every $f \in \vV^*$, and for every $\bsr  \in U\times U$, the parametric weak  solution $\vu=\vu(\cdot,\bsr)$ of  problem \eqref{para weak} satisfies the estimate
\begin{equation}\label{a priori}
  \|\vu\|_{\vV} + \Lambda \|\nabla \cdot \vu\| 
 \le C\,  \|\vf\|_{\vV^*}\,.
\end{equation}
Furthermore, if $\Omega$ is a convex domain (or if $\Omega$ has a smooth boundary) and   $\vf \in {\bf L}^2(\Omega)$, then  $\vu(\cdot,\bsr) \in \vV \cap {\bf H}$, which is a strong solution of \eqref{eq:L1}, and the following estimate holds true 
\begin{equation}\label{a priori H2}
  \|\vu\|_{\bf H} +  \|(\lambda+\mu) \nabla \cdot \vu\|_{\vV} 
 \le  C\,\|\vf\|,\quad {\rm for~every}~~\bsr  \in U\times U\,.
\end{equation} 
The  constant $C$ in \eqref{a priori}  depends on  $\Omega$, $\mu$ and $\widehat\lambda$, while in \eqref{a priori H2}, $C$ depends (linearly) on  $\|\nabla \mu(\cdot,\bsy)\|_{L^\infty(\Omega)}$ and $\|\nabla \widehat \lambda(\cdot,\bsz)\|_{L^\infty(\Omega)}$  as well. In both \eqref{a priori} and \eqref{a priori H2}, $C$  is independent of $\Lambda$.
\end{theorem}
\begin{proof}
From the coercivity property in \eqref{eq: coer B} and the weak formulation in \eqref{para weak},  we have   
\begin{equation}\label{eq: H1 bound of u}
C \mu_{\min}\| \vu\|_{\vV}^2\le  2 \mu_{\min}  \|\veps(\vu)\|^2 \le   \calB(\bsr;\vu,\vu)=\ell(\vu) \le \|\vf\|_{\vV^*} \|\vu\|_{\vV},
\end{equation} 
for every $\bsr \in U\times U$. Since $\vu(\cdot,\bsr) \in \vV,$  by \cite[Theorem 3.1 and (3.2)]{ArnoldScottVogelius1988}, there exists $\vu^* \in \vV$ such that $\nabla \cdot \vu= \nabla \cdot \vu^*$ and $\|\vu^*\|_{\vV}\le C\|\nabla \cdot \vu\|$. 
Using this in the weak formulation  \eqref{para weak}, after some calculations, we reach  
\begin{multline}\label{eq:u*}
  \Lambda \widehat \lambda_{\min}  \|\nabla \cdot \vu\|^2 \le \int_\Omega \lambda \nabla \cdot \vu \nabla \cdot \vu^* \,\,d\bsx
 = \ell(\vu^*)-\int_\Omega 2\mu \veps(\vu):\veps(\vu^*)\,\,d\bsx\\
\le C\|\vf\|_{\vV^*} \|\vu^*\|_{\vV} +2\mu_{\max} \|\veps(\vu)\|\,\|\veps(\vu^*)\|
\le C\Big(\|\vf\|_{\vV^*}  +2\mu_{\max} \|\vu\|_{\vV}\Big)\|\nabla \cdot \vu\|\,.    
\end{multline}
Thus,  $\Lambda \widehat \lambda_{\min}  \|\nabla \cdot \vu\| \le C\Big(\|\vf\|_{\vV^*}  +2\mu_{\max} \|\vu\|_{\vV}\Big)$. Combining this estimate with the inequality $\| \vu\|_{\vV}  \le \frac{C}{\mu_{\min}}\|\vf\|_{\vV^*}$, which is due to 
\eqref{eq: H1 bound of u}, completes the proof of the regularity estimate in \eqref{a priori}. 

We turn now to the case of a convex domain $\Omega$. Since  $\vf \in {\bf L}^2(\Omega)$ and since the bilinear operator $\calB(\bsr;\cdot,\cdot)$ is $\vV-$elliptic for every $\bsr \in U\times U$, the solution $\vu(\cdot,\bsr)$ of \eqref{para weak} is in the  space   $\bH$ owing to the elliptic regularity theory, and consequently, $\vu(\cdot,\bsr)$ is a strong solution of \eqref{eq:L1}. 
For more details about strongly elliptic systems and regularity, see \cite[Theorem~4.18]{McLean2000}.  

To show the regularity estimate in \eqref{a priori H2}, we use first  the identity 
 $\nabla \cdot (g {\bf A})=g \nabla \cdot {\bf A} +\nabla g\cdot {\bf A}$ (with $g \in W^{1,\infty}(\Omega)$, ${\bf A}=[a_{ij}]$ is an $m$-by-$m$ matrix-valued function, and so  $\nabla g\cdot {\bf A}=[\nabla g \cdot {\bf A}_i]$ is an $m$-by-$1$ column vector where ${\bf A}_i$ is the $i$th row of ${\bf A}$), and  notice that 
\[\nabla \cdot(\vsigma(\vu,\lambda,\mu))= \lambda\nabla \cdot\Big(\nabla\cdot \vu\, I\Big) 
+ 2\mu \nabla \cdot\veps(\vu)+\nabla  \lambda \cdot\Big(\nabla\cdot \vu\, I\Big) 
+ 2\nabla \mu  \cdot\veps(\vu)\,.\]
However  $\nabla \cdot (\nabla \cdot \vu I)=\nabla (\nabla \cdot \vu)$ and $2 \nabla\cdot \veps(\vu) = \Delta \vu +\nabla(\nabla \cdot \vu)$, so problem \eqref{eq:L1} can be reformulated as: for $\bsr \in U\times U,$  
\begin{equation}\label{Strong form 0} -\mu\Delta  \vu -(\lambda+\mu)\nabla(\nabla \cdot \vu)= \vf+\nabla \lambda\, \Big(\nabla\cdot \vu\Big) 
+ 2\nabla \mu  \cdot \veps(\vu)=:\mu\,{\bf G}.
\end{equation}
For later use,  it is not hard  to observe that 
\begin{equation}\label{Prestimate}
\|\nabla(\nabla \cdot \vu)\|\le \|\eta^{-1} \Delta  \vu\| +\|\eta^{-1}{\bf G}\|~~{\rm with}~~\eta=\frac{ \lambda+\mu}{\mu}=\frac{\lambda}{\mu}+1\,.
\end{equation}
By dividing now both sides of \eqref{Strong form 0}  by $\mu$ and  letting  $p=-\eta \nabla \cdot \vu  \in H^1(\Omega)$, after some simplifications,  we obtain 
\begin{equation}\label{Strong form}
-\Delta \vu +\nabla p= {\bf \widetilde G}:=\vf/\mu + 2\nabla \mu  \cdot\veps(\vu)/\mu+ \nabla \cdot \vu\,\Big( \lambda/\mu^2\Big)\nabla \mu\,.\end{equation} 
We borrow now the idea used in \cite[Lemma 2.2]{BrennerSung1992}, specifically (2.19) and (2.20) therein, 
\begin{equation}\label{u*}
\exists ~    \vu^* \in {\bf H}\cap \vV~{\rm such~that}~~\nabla \cdot \vu= \nabla \cdot \vu^*~{\rm and}~ \|\vu^*\|_{\bf H}\le C_0\|\nabla \cdot \vu\|_{\vV},
\end{equation}
 where the generic constant $C_0$ depends on $\Omega$ only. (This can be guaranteed by using   \cite[Theorem 3.1 and (3.2)]{ArnoldScottVogelius1988}.)  Hence, by setting $\vw=\vu-\vu^*$, problem \eqref{Strong form} can be reformulated in (homogeneous) Stokes form as: find $\vw \in {\bf H}\cap \vV$ and $p \in H^1(\Omega)$ such that 
\begin{equation}\label{Stokes form}
-\Delta \vw +\nabla p= {\bf \widetilde G}-\Delta \vu^*,\quad{\rm with}~~\nabla \cdot \vw=0\,.
\end{equation} 
Thanks to the regularity estimates for the Stokes problem in \cite[Theorem 2]{KelloggOsborn1976}, we have 
\[\|\vw\|_{\bf H}+\|\nabla p\|\le C_0\|{\bf \widetilde G}-\Delta \vu^*\|,~~\text{assuming that $\Omega$ is convex.}\]
    If  $\Omega$ is a ${\mathcal C}^{1,1}$ bounded domain in $\R^d$ (with $d\in \{2,3\}$), a similar result was proved in \cite[Theorem~3]{AmroucheGirault1991}. Furthermore, for a ${\mathcal C}^2$ class domain,  we refer the reader to \cite[Theorem IV.6.1]{Galdi2011} for similar regularity results.  Recall that  $\vu=\vw+\vu^*$ and $ p=-\eta\nabla \cdot \vu $, thence  \[\|\vu\|_{\bf H}+\|\nabla(\eta\nabla \cdot \vu)\|\le C_0\Big(\|{\bf \widetilde G}\|+\|\vu^*\|_{\bf H}\Big)\,.\]
  However, from \eqref{u*}, the regularity estimate in \eqref{a priori}, and  the pre-estimate in \eqref{Prestimate},  
  we have 
  \[ \|\vu^*\|_{\bf H} \le C_0\|\nabla \cdot \vu\|_{\vV}\le C_0\|\nabla(\nabla \cdot \vu)\|+ C \Lambda^{-1}\|\vf\|
  \le C_0(\|\eta^{-1}\Delta \vu\| +\|\eta^{-1}{\bf G}\|)+C \Lambda^{-1}\|\vf\|,
  \] and consequently,  
\[\|\vu\|_{\bf H}+\|\nabla(\eta\nabla \cdot \vu)\|
\le C_0(\|\eta^{-1}\Delta \vu\| +\|\eta^{-1}{\bf G}\|)+ C\Lambda^{-1}\|\vf\|+C\|{\bf \widetilde G}\| \,.\]
From the estimate in  \eqref{a priori} together with the embedding of ${\bf L}^2(\Omega)$ into $\vV^*$, we have 
\begin{equation*}
     \|\eta^{-1}{\bf G}\| \le C\Big(1+\|\nabla\widehat \lambda\|_{L^\infty(\Omega)}/\widehat \lambda_{\min}\Big)\|\vf\|~~{\rm and}~~ 
     \|{\bf \widetilde G}\| \le C\Big(1+\widehat \lambda_{\max}/\widehat \lambda_{\min}\Big)\|\vf\|\,.
\end{equation*}
Combining the above two equations, we conclude that
\[  \|\vu\|_{\bf H}+\|\nabla(\eta\nabla \cdot \vu)\|
\le C_0\|\eta^{-1}\Delta \vu\| 
+C\|\vf\|\le C_0\frac{\mu_{\max}}{\widehat \lambda_{\min} \Lambda} \|\Delta \vu\| 
+C\|\vf\|
\le \frac{1}{2}\|\vu\|_{\bf H}
+C\|\vf\|,\]
for a sufficiently large $\Lambda$. Canceling the common terms yields  
\begin{equation}\label{SL}
    \|\vu\|_{\bf H}+2\| \nabla(\eta\nabla \cdot \vu)\|
\le  C\|\vf\|\,.
\end{equation}
To finalize the proof of \eqref{a priori H2}, it is sufficient to show  that  
\begin{equation}\label{last bound}
\|\nabla((\lambda+\mu)\nabla \cdot \vu)\|\le C\|\vf\|\,,
\end{equation}
and then, merging this with \eqref{SL} and the regularity estimate in \eqref{a priori}. In fact, the bound in \eqref{last bound} is a consequence of the inequality
\[\|\nabla((\lambda+\mu)\nabla \cdot \vu)\|=
\|\nabla(\mu\eta\nabla \cdot \vu)\|\le C\|\nabla(\eta\nabla \cdot \vu)\|+C\|\eta\nabla \cdot \vu\|\le C\|\nabla(\eta\nabla \cdot \vu)\|+C\Lambda \|\nabla \cdot \vu\|\,,\]
 and the achieved estimates in \eqref{SL} and \eqref{a priori}.
\end{proof}
\subsection{Regularity in random coefficients}
For the QMC error bounds, we prove in the next theorem  a mixed partial derivatives estimate  of the parametric displacement  $\vu$ with respect to the random variables $y_j$ and $z_k$. 
Here, ${\calS}$ denotes the set of (multi-index) infinite vectors  $\balpha=(\alpha_j)_{j\ge 1}$ with nonnegative integer entries such that  $|\balpha|:=\sum_{j\ge 1} \alpha_j<\infty.$  For  $\balpha=(\alpha_j)_{j\ge 1}$  and $\bbeta=(\beta_j)_{j\ge 1}$ in  $\calS,$ the mixed partial derivative with respect to $\bsy$ and $\bsz$ is denoted by   
\[\partial_{\bsr}^{\balpha,\bbeta}:=\partial_{\bsz}^{\bbeta} \partial_{\bsy}^{\balpha}  =\frac{\partial^{|\bbeta|}}{\partial_{z_1}^{\beta_1}\partial_{z_2}^{\beta_2}\cdots}\frac{\partial^{|\balpha|}}{\partial_{y_1}^{\alpha_1}\partial_{y_2}^{\alpha_2}\cdots}, \quad \bsr=(\bsy,\bsz)\,. 
\]
 
\begin{theorem}\label{lem: vu bound y z}
Under assumptions \eqref{ass A1} and \eqref{ass A5}, for every $\vf \in \vV^*$, every $\bsr \in  U\times U$ and every $\balpha,\bbeta \in \calS$, the solution $\vu(\cdot,\bsr)$ of the parametric weak problem \eqref{para weak} satisfies
\begin{equation}\label{mixed estimate}
\big\|\veps\big(\partial_{\bsy,\bsz}^{\balpha,\bbeta} \vu(\cdot,\bsy,\bsz)\big)\big\|+
\Lambda \big\|\veps\big(\partial_{\bsy,\bsz}^{\balpha,\bbeta} \vu(\cdot,\bsy,\bsz)\big)\big\|\le 
(|\balpha|+|\bbeta|)! \
\widetilde {\bf b}^{\balpha} \widehat {\bf b}^{\bbeta}\Big(\|\veps(\vu)\|+\Lambda \|\nabla\cdot \vu\|\Big) ,
\end{equation} 
where 
\[\widetilde {\bf b}^{\balpha}=\prod_{i\ge 1} (\widetilde b_i)^{\alpha_i},\quad 
\widetilde b_j= \Big(\frac{2C_1}{\widehat \lambda_{\min}}\Big(1 + \frac{\mu_{\max}}{\mu_{\min}}\Big)+\frac{1}{\mu_{\min}}\Big)\|\psi_j\|_{L^\infty(\Omega)},\]
and
\[\widehat {\bf b}^{\bbeta}=\prod_{i\ge 1} (\widehat b_i)^{\beta_i},\quad  \widehat b_j=\Big(\frac{1}{\widehat \lambda_{\min}}\Big(1+\sqrt{2}C_1 \frac{\mu_{\max}}{\mu_{\min}}\Big)+\frac{1}{\sqrt{2}\mu_{\min}}\Big)\|\phi_j\|_{L^\infty(\Omega)}\,.\]
The constant $C_1$ occurring in the proof depends on $\Omega$ only. Consequently,
 \begin{equation}\label{mixed estimate 2}
 \|\partial_{\bsy,\bsz}^{\balpha,\bbeta} \vu(\cdot,\bsy,\bsz)\|_{\vV} \le 
C(|\balpha|+|\bbeta|)! \
\widetilde {\bf b}^{\balpha} \widehat {\bf b}^{\bbeta}\|\vf\|_{\vV^*},
\end{equation}
where the constant $C$ depends on   $\Omega$, $\mu$ and $\widehat\lambda$, but   is independent of $\Lambda$.
\end{theorem}
\begin{proof} Differentiating both sides of \eqref{para weak} with respect to the variables $y_j$ and $z_k$, we find the following recurrence after some calculations 
\begin{equation}\label{recurrence}
 \calB(\bsr;\partial_{\bsy,\bsz}^{\balpha,\bbeta}\vu, \vv) = -\int_\Omega \Big[2\sum_{\balpha} \alpha_j\psi_j\, \veps(\partial_{\bsy,\bsz}^{\balpha-{\bf e}_j,\bbeta}\vu):\veps(\vv) +\Lambda \sum_{\bbeta}\beta_k\phi_k \nabla \cdot \partial_{\bsy,\bsz}^{\balpha,\bbeta-{\bf e}_k} \vu \nabla \cdot \vv\Big] \,d\bsx,
\end{equation}
where $\sum_{\balpha}=\sum_{j,\alpha_j\ne 0}$ (that is, the sum over the nonzero indices of $\balpha$), and   ${\bf e}_i \in \calS$ denotes the multi-index with entry $1$ in position $i$ and zeros elsewhere. Choose $\vv=\partial_{\bsy,\bsz}^{\balpha,\bbeta}\vu$ in \eqref{recurrence}, then from the coercivity property in \eqref{eq: coer B}  we have   
\begin{multline*}
2\mu_{\min} \|\veps(\partial_{\bsy,\bsz}^{\balpha,\bbeta}\vu)\|^2 \le   \calB(\bsr;\partial_{\bsy,\bsz}^{\balpha,\bbeta}\vu,\partial_{\bsy,\bsz}^{\balpha,\bbeta}\vu)\le 
2\sum_{\balpha}\alpha_j\|\psi_j\|_{L^\infty(\Omega)} \|\veps(\partial_{\bsy,\bsz}^{\balpha-{\bf e}_j,\bbeta} \vu)\|\,
\|\veps(\partial_{\bsy,\bsz}^{\balpha,\bbeta} \vu)\|\\
+ \Lambda \sum_{\bbeta}\beta_k\, \|\phi_k\|_{L^\infty(\Omega)}
\|\nabla \cdot \partial_{\bsy,\bsz}^{\balpha,\bbeta-{\bf e}_k} \vu\|\,\|\nabla \cdot \partial_{\bsy,\bsz}^{\balpha,\bbeta} \vu\|\,,
\end{multline*} 
and thus, using the inequality $\|\nabla \cdot \vv\|\le \sqrt{2}\,\|\veps(\vv)\|$ ($2$ is the physical dimension) gives 
\begin{equation}\label{veps bound} \|\veps(\partial_{\bsy,\bsz}^{\balpha,\bbeta}\vu)\| \le 
\frac{1}{\mu_{\min}}\Big(\sum_{\balpha}\alpha_j\|\psi_j\|_{L^\infty(\Omega)} \|\veps(\partial_{\bsy,\bsz}^{\balpha-{\bf e}_j,\bbeta} \vu)\|
+ \Lambda  \sum_{\bbeta}\beta_k\, \frac{\|\phi_k\|_{L^\infty(\Omega)}}{\sqrt{2}}
\|\nabla \cdot \partial_{\bsy,\bsz}^{\balpha,\bbeta-{\bf e}_k} \vu\|\Big)\,,
\end{equation}
for every $\bsr \in U\times U$. To proceed in our proof, we argue as in \eqref{eq:u*} and deduce that  there exists $(\partial_{\bsy,\bsz}^{\balpha,\bbeta}\vu)^* \in \vV$ such that 
$\nabla \cdot \partial_{\bsy,\bsz}^{\balpha,\bbeta} \vu= \nabla \cdot (\partial_{\bsy,\bsz}^{\balpha,\bbeta}\vu)^*$ and $\|(\partial_{\bsy,\bsz}^{\balpha,\bbeta}\vu)^*\|_{\vV}\le C_1\|\nabla \cdot \partial_{\bsy,\bsz}^{\balpha,\bbeta}\vu\|$. Using this and the inequality $\|\veps((\partial_{\bsy,\bsz}^{\balpha,\bbeta}\vu)^*)\|\le \|(\partial_{\bsy,\bsz}^{\balpha,\bbeta}\vu)^*\|_{\vV}$ in the weak formulation  \eqref{recurrence} with $\vv=(\partial_{\bsy,\bsz}^{\balpha,\bbeta}\vu)^*$, we obtain   
\begin{multline*}\label{eq:alphabetau*}
 \Lambda \widehat \lambda_{\min}  \|\nabla \cdot \partial_{\bsy,\bsz}^{\balpha,\bbeta}\vu\|^2 \le \int_\Omega \lambda \nabla \cdot \partial_{\bsy,\bsz}^{\balpha,\bbeta}\vu \nabla \cdot (\partial_{\bsy,\bsz}^{\balpha,\bbeta}\vu)^* \,\,d\bsx
\le 2\sum_{\balpha}\alpha_j \|\psi_j\|_{L^\infty(\Omega)} \|\veps(\partial_{\bsy,\bsz}^{\balpha-{\bf e}_j,\bbeta} \vu)\|\,
\|\veps((\partial_{\bsy,\bsz}^{\balpha,\bbeta}\vu)^*)\|\\
+ \Lambda  \sum_{\bbeta}\beta_k\, \|\phi_k\|_{L^\infty(\Omega)}
\|\nabla \cdot \partial_{\bsy,\bsz}^{\balpha,\bbeta-{\bf e}_k} \vu\|\,\|\nabla \cdot (\partial_{\bsy,\bsz}^{\balpha,\bbeta}\vu)^*\|
+2\mu_{\max} \|\veps(\partial_{\bsy,\bsz}^{\balpha,\bbeta}\vu)\|\,\|\veps((\partial_{\bsy,\bsz}^{\balpha,\bbeta}\vu)^*)\|\\
\le \Big(2C_1\sum_{\balpha}\alpha_j \|\psi_j\|_{L^\infty(\Omega)} \|\veps(\partial_{\bsy,\bsz}^{\balpha-{\bf e}_j,\bbeta} \vu)\|
+ \Lambda \sum_{\bbeta}\beta_k\, \|\phi_k\|_{L^\infty(\Omega)}
\|\nabla \cdot \partial_{\bsy,\bsz}^{\balpha,\bbeta-{\bf e}_k} \vu\|\\
+2C_1\mu_{\max} \|\veps(\partial_{\bsy,\bsz}^{\balpha,\bbeta}\vu)\|\Big)\,\|\nabla \cdot \partial_{\bsy,\bsz}^{\balpha,\bbeta}\vu\|\,.    
\end{multline*}
Dividing both sides by $\|\nabla \cdot \partial_{\bsy,\bsz}^{\balpha,\bbeta}\vu\|$ and using \eqref{veps bound}  yield 
\begin{multline*}
    \Lambda  \|\nabla \cdot \partial_{\bsy,\bsz}^{\balpha,\bbeta}\vu\|\le \frac{2C_1}{\widehat \lambda_{\min}}\Big(1 + \frac{\mu_{\max}}{\mu_{\min}}\Big) \sum_{\balpha}\alpha_j \|\psi_j\|_{L^\infty(\Omega)} \|\veps(\partial_{\bsy,\bsz}^{\balpha-{\bf e}_j,\bbeta} \vu)\|\\
+ \frac{\Lambda}{\widehat \lambda_{\min}}\Big(1+\sqrt{2}C_1 \frac{\mu_{\max}}{\mu_{\min}}\Big) \sum_{\bbeta}\beta_k\, \|\phi_k\|_{L^\infty(\Omega)}
\|\nabla \cdot \partial_{\bsy,\bsz}^{\balpha,\bbeta-{\bf e}_k} \vu\|\,. 
\end{multline*}      
Combining this estimate with~\eqref{veps bound} leads to 
\begin{equation}\label{estimate 1}
\|\veps(\partial_{\bsy,\bsz}^{\balpha,\bbeta}\vu)\|+ \Lambda \|\nabla \cdot \partial_{\bsy,\bsz}^{\balpha,\bbeta}\vu\|
\le 
\sum_{\balpha}\alpha_j \widetilde b_j \|\veps(\partial_{\bsy,\bsz}^{\balpha-{\bf e}_j,\bbeta} \vu)\|
+ \Lambda  \sum_{\bbeta}\beta_k\,\widehat b_k
\|\nabla \cdot \partial_{\bsy,\bsz}^{\balpha,\bbeta-{\bf e}_k} \vu\|\,. 
\end{equation}
To proceed in our proof, we use the induction hypothesis on $n:=|\balpha+\bbeta|$. From \eqref{estimate 1}, it is clear that \eqref{mixed estimate} holds true when $|\balpha+\bbeta|=1$. Now, assume that \eqref{mixed estimate} is true for $|\balpha+\bbeta|=n$, and the task is to show it for $|\balpha+\bbeta|=n+1.$ 

From \eqref{estimate 1} and the induction hypothesis, we have 
\begin{multline*}
    \|\veps(\partial_{\bsy,\bsz}^{\balpha,\bbeta} \vu)\|+ \Lambda \|\nabla \cdot \partial_{\bsy,\bsz}^{\balpha,\bbeta}\vu\|\le 
 n! \Big(\sum_{\balpha}\alpha_j \widetilde b_j  
\widetilde {\bf b}^{\balpha-{\bf e}_j} \widehat {\bf b}^{\bbeta}\|\veps(\vu)\|+ \Lambda \sum_{\bbeta}\beta_k\, \widehat b_k
\widetilde {\bf b}^{\balpha} \widehat {\bf b}^{\bbeta-{\bf e}_k}\|\nabla \cdot\vu\|\Big)\\
\le  n!\widetilde {\bf b}^{\balpha} \widehat {\bf b}^{\bbeta}\Big(\sum_{\balpha}\alpha_j  + \sum_{\bbeta}\beta_k\Big)(\|\veps(\vu)\|+\Lambda \|\nabla \cdot \vu\|)\,.
\end{multline*}
Since $\sum_{\balpha}\alpha_j  + \sum_{\bbeta}\beta_k=n+1$, the proof of \eqref{mixed estimate} is completed.  

Finally, since $\|\veps(\partial_{\bsy,\bsz}^{\balpha,\bbeta} \vu(\cdot,\bsy,\bsz))\| \ge C\,\|\partial_{\bsy,\bsz}^{\balpha,\bbeta} \vu(\cdot,\bsy,\bsz)\|_{\vV}$ (by Korn's inequality) and since $\|\veps(\vu)\|+\Lambda \|\nabla \cdot \vu\|\le \|\nabla \vu\|+\Lambda \|\nabla \cdot \vu\|\le C\|\vf\|_{\vV^*}$ (by \eqref{a priori}), we obtain \eqref{mixed estimate 2} from \eqref{mixed estimate}. 
 \end{proof}

\section{A truncated problem and error estimates}\label{Sec: Truncation}
In this section, we  investigate the error from truncating the first and second sums in \eqref{KLexpansion} at $s_1$ and $s_2$ terms, respectively, for some $s_1,s_2 \in \N.$  First, let us define the truncated weak formulation problem: for every $\bsr^\vs \in U \times U,$ find $\vu_\vs(\cdot, \bsr^\vs)  \in \vV$,    such that
\begin{equation}\label{truncated weak}
  \calB  (\bsr^\vs;\vu_\vs(\cdot,\bsr^\vs),\vv) = \ell(\vv), \qquad \forall ~\vv \in \vV.
\end{equation}
By  Theorem \ref{thm: unique solution}, the variational  problem \eqref{truncated weak} has a unique solution. In the next theorem, we estimate $\calL(\vu-\vu_\vs)$. This is needed for measuring the  QMC finite element error in \eqref{combine}. 
 \begin{theorem}\label{Truncating error}
Let $\vu$ and $\vu_\vs$ be the solutions of the parametric weak problems
\eqref{para weak} and \eqref{truncated weak}, respectively.  Assume that $|\calL(\vw)|\le \|\calL\|_{\vV^*}\|\vw\|_\vV$ for any $\vw \in \vV$. Under Assumption \eqref{ass A1}, for every $\vf \in \vV^*$, every $\bsr \in U\times U$, and
every $\vs \in \N^2$, we have
\begin{equation}\label{convergence calL u-us}
|\calL(\vu(\cdot, \bsr))-\calL(\vu_\vs(\cdot,\bsr^\vs))| 
\le C\,\widetilde C\,   \|\vf\|_{\vV^*} \|\calL\|_{\vV^*},
\end{equation}
where $C$ depends on $\Omega,$ $\mu$, and $\widehat \lambda$, and 
$\widetilde C= 
 \sum_{j \ge s_1+1} \|\psi_j\|_{L^\infty(\Omega)}
+ \sum_{j \ge s_2+1} \|\phi_j\|_{L^\infty(\Omega)}\,.$
 Moreover, if  $\sum_{j \ge s_1+1} \|\psi_j\|_{L^\infty(\Omega)} \le C_p s_1^{1-1/p}$ and $\sum_{j \ge s_2+1} \|\phi_j\|_{L^\infty(\Omega)} \le C_q s_2^{1-1/q}$ for some $0<p,\,q<1,$ then $\widetilde C \le  C_{p,q} \Big(s_1^{1-\frac{1}{p}}+ s_2^{1-\frac{1}{q}}\Big).$ 
 Note that, by using the Stechkin inequality, a similar bound can be observed  if  
 \eqref{ass A5}~and \eqref{ass A3} are satisfied instead.  
\end{theorem}
\begin{proof}
From the variational formulations in \eqref{para weak} and \eqref{truncated weak}, we have  
\begin{equation}\label{weak difference}
    \calB  (\bsr^\vs;\vu_\vs(\cdot,\bsr^\vs)-\vu(\cdot,\bsr),\vv)=\calB  (\bsr-\bsr^\vs;\vu(\cdot,\bsr),\vv)\,.
\end{equation}
By following the steps in \eqref{eq: bounded} and using the achieved estimate in \eqref{a priori}, we have 
\begin{multline*}
|\calB  (\bsr-\bsr^\vs;\vu(\cdot,\bsr),\vv)|
\le C\Big(\sum_{j \ge s_1+1} \|\psi_j\|_{L^\infty(\Omega)}
 \,\|\vu(\cdot,\bsr)\|_{\vV}
+\Lambda \sum_{j \ge s_2+1} \|\phi_j\|_{L^\infty(\Omega)}\, \|\nabla \cdot \vu(\cdot,\bsr)\|\Big)\,\|\vv\|_{\vV}\\
\le C\,\widetilde C\Big(\|\vu(\cdot,\bsr)\|_{\vV}
+\Lambda\, \|\nabla \cdot \vu(\cdot,\bsr)\|\Big)\,\|\vv\|_{\vV}\le  C\,\widetilde C  \| \vf \|_{\vV^*}\|\vv\|_{\vV}\,.
\end{multline*}

Now, by mimicking the proof of  \eqref{a priori} with \eqref{weak difference} instead of the weak formulation in  \eqref{para weak}, and the above estimate instead of $|\ell(\vv)|\le \|\vf\|_{\vV^*}\|\vv\|_{\vV},$ we obtain 
\begin{equation}\label{middle bound}
\|\vu(\cdot,\bsr) - \vu_\vs(\cdot,\bsr^\vs)\|_{\vV} 
 \le C\,\widetilde C  \|\vf\|_{\vV^*}\,.   
\end{equation}
From the linearity and boundedness properties of $\calL$, 
\[  |\calL(\vu(\cdot,\bsr)-\vu_\vs(\cdot,\bsr^\vs))|
 \le C\|\vu(\cdot,\bsr)-\vu_\vs(\cdot,\bsr^\vs)\|_{\vV}\|\calL\|_{\vV^*}\,,\]
 and hence, to complete the proof, we insert \eqref{middle bound} in this equation. 
 \end{proof}
\section{Nonconforming FEM}\label{Sec: FEM}
In this section, we develop and analyze a low-order nonconforming Galerkin FEM  for approximating the solution of the model problem~\eqref{para weak} (and consequently, of problem \eqref{eq:L1}) over the physical domain~$\Omega$, and also investigate the numerical errors. We introduce a family of regular quasi-uniform triangulations  $\mathcal{T}_h$ of the domain $\overline{\Omega}$ and let $h=\max_{K\in \mathcal{T}_h}(h_K)$, where $h_{K}$ denotes the diameter of the element $K$. For convenience, we introduce the following notations: $\calE=\{E_1,E_2,\ldots, E_{q_h}\}$ is the set of all edges   of the interelement boundaries of the elements in 
${\mathcal T}_h$, while $\calF=\{F_1,F_2,\ldots, F_{b_h}\}$ is the set of all remaining edges that are on the boundary of $\Omega.$ The corresponding  midpoints of the edges in $\calE$ and $\calF$ are denoted by $\xi_1,\xi_2,\ldots,\xi_{q_h}$ and $\eta_1,\eta_2,\ldots,\eta_{b_h}$, respectively. The space of linear polynomials is denoted by $P_1.$  The nonconforming finite element  space is then defined by \begin{multline*}
     \vV_h=\{\vw_h \in {\bf L}^2(\Omega):~\vw_h |_K \in [P_1]^2,~\forall~K \in \mathcal{T}_h, \\
~\vw_h {\rm~ is~continuous~at}~\xi_i~{\rm and}~ \vw_h(\eta_j)=0,~{\rm for}~~1\le i\le q_h,~1\le j\le b_h\}.
\end{multline*}


Developing a well-posed nonconforming FEM that maintains the symmetric structure of the variational formulation of problem \eqref{eq:L1} (which is important at both a theoretical and practical level) becomes more technically challenging when the Lam\'e parameter~$\mu$ is
allowed to depend on the spatial variable~$\bsx$, so an innovative touch is needed here. To elaborate, from the definitions of $\vsigma$ and $\veps,$ and by using the identity $ \nabla \cdot (\mu \nabla \vu)^T=   \nabla \mu  \cdot(\nabla \vu)^T+\mu \nabla(\nabla \cdot \vu)$ (which holds true provided that the mixed first derivatives of the entries of $\vu$ are commutative),  we notice that 
\begin{multline*}\nabla \cdot \vsigma(\vu) = \nabla (\lambda \nabla\cdot \vu)+ \nabla \cdot (\mu \nabla \vu) + \nabla \cdot (\mu \nabla \vu)^T\\
  = \nabla (\lambda \nabla\cdot \vu) +\nabla \cdot (\mu \nabla \vu) +  \nabla \mu  \cdot(\nabla \vu)^T+\mu \nabla(\nabla \cdot \vu)\,.\end{multline*} 
Therefore, since $\mu \nabla(\nabla \cdot \vu)=\nabla(\mu\nabla \cdot \vu)-\nabla\mu \nabla \cdot \vu$,   we may reformulate  \eqref{eq:L1} by rewriting the principle part in divergence form as: for every $\bsr \in U\times U$ and every $\bsx \in \Omega,$
 \begin{equation}\label{strong new}
    -\nabla \cdot(\mu \nabla  \vu) -\nabla (( \lambda+\mu)\nabla \cdot \vu)+\nabla \mu\, \nabla\cdot \vu - \nabla \mu  \cdot(\nabla \vu)^T=\vf\,.
\end{equation} 
 When the Lam\'e parameters $\mu$ and $\lambda$ are constants, \eqref{strong new} reduces to the following strong form: $-\mu\Delta  \vu -( \lambda+\mu)\nabla (\nabla \cdot \vu)=\vf$.
 The weak formulation and the nonconforming FEM can then be defined accordingly; see for example  \cite{BrennerSung1992}. For the weak formulation of \eqref{strong new} when  $\mu$ and $\lambda$ vary with~$\bsx$, we take the ${\bf L}^2(\Omega)$ inner product with a test function $\vv\in \vV$, then use the divergence theorem and the  homogeneous Dirichlet boundary condition on $\vv$. The variational form of problem \eqref{strong new}   is then defined as: find $\vu(\cdot, \bsr) \in  \vV$ such that 
\begin{equation}\label{para weak new}
\widehat  \calB(\bsr;\vu, \vv) = \ell(\vv), \quad \text{for all} \quad \vv \in  \vV,
\end{equation}
where
\[ \widehat  \calB(\bsr;\vu, \vv) := \int_\Omega [\mu\, \nabla \vu:\nabla \vv +(\mu+\lambda) \nabla \cdot \vu \nabla \cdot \vv +(\nabla \mu\, \nabla\cdot \vu - \nabla \mu  \cdot(\nabla \vu)^T)\cdot\vv] \,d\bsx\,.\]
Here, the last term $\int_\Omega (\nabla \mu\, \nabla\cdot \vu - \nabla \mu  \cdot(\nabla \vu)^T)\cdot \vv \,d\bsx$ of the integrand equals
\[\int_\Omega \nabla \mu\cdot ( \vk_{u_1} v_2+\vk_{v_1} u_2) \,d\bsx,~~{\rm with}~~\vu=\begin{bmatrix}u_1\\u_2\end{bmatrix},~~\vv=\begin{bmatrix}v_1\\v_2\end{bmatrix},~~~\bsx=(x_1,x_2),~~{\rm and}~~\vk_g=\begin{bmatrix}-g_{x_2}\\g_{x_1}\end{bmatrix},\]
where the sub index $x_i$ denotes the partial derivative with respect to $x_i$ for $i=1,2.$ Then, by applying again the divergence theorem and using the  homogeneous Dirichlet boundary conditions  $u_2=0$ on $\Gamma$, assuming that the mixed first order derivatives of $\mu$  exist and commute,  we reformulate the operator $\widehat  \calB(\bsr;\cdot, \cdot)$ in a symmetric form as: 
\begin{equation}\label{eq: bilinear ref}
 \widehat  \calB(\bsr;\vu, \vv) := \int_\Omega [\mu\, \nabla \vu:\nabla \vv +(\mu+\lambda) \nabla \cdot \vu \nabla \cdot \vv +\nabla \mu\cdot ( \vk_{u_1} v_2+\vk_{v_1} u_2)] \,d\bsx\,.\end{equation}
Alternatively, $\widehat  \calB(\bsr;\vu, \vv)$ can be extracted  from $\calB(\bsr;\vu, \vv)$ by reformulating the latter  as 
\[\calB(\bsr;\vu, \vv) = \int_\Omega [\mu\, \nabla \vu:\nabla \vv +(\mu+\lambda) \nabla \cdot \vu \nabla \cdot \vv -\mu\,\vk_{u_1} \cdot \nabla v_2-\mu\,\vk_{v_1}\cdot \nabla u_2)] \,d\bsx\,,\]
then applying the divergence theorem to the last two terms to obtain $\widehat  \calB(\bsr;\vu, \vv)$, assuming that the mixed derivatives of $u_1$ and of $v_1$ exist and commute.    

 Motivated by \eqref{para weak new},  we define the nonconforming FEM as: find $\vu_h(\cdot, \bsr) \in  \vV_h$ such that \begin{equation}\label{FEM new}
\widehat  \calB_h(\bsr;\vu_h, \vv_h) = \ell(\vv_h), \quad \text{for all $\vv_h \in  \vV_h$ and for every $\bsr \in U\times U,$}
\end{equation}
where the {\em broken}  bilinear form $\widehat\calB_{h}(\bsr;\cdot,\cdot)$ is defined as $\widehat\calB(\bsr;\cdot,\cdot)$
 in \eqref{eq: bilinear ref}, but with $\nabla_h$ and $\nabla_h \cdot$ in place of $\nabla$ and $\nabla\cdot$, respectively. The discrete gradient operator $\nabla_h$ is defined as:  $\nabla_h \vv|_K= \nabla (\vv|_K)$ for any $K\in \mathcal T_h$, and the discrete divergence operator $\nabla_h\cdot$  is defined similarly.

The existence and uniqueness of the  solution $\vu_h(\cdot,\bsr)$ of \eqref{FEM new} is guaranteed if the operator $\widehat\calB_{h}(\bsr;\cdot,\cdot)$ is positive-definite on $\vV_h$. To show this,  we introduce  the following broken norms on the finite dimensional space $\vV_h$: for $\vv \in V_h,$  
\[ \|\vv\|_h^2 =\sum_{K \in \mathcal T_h}\int_K \nabla\vv:\nabla\vv \,d\bsx~~{\rm and}~~
\|\vv\|_{\calB_h}^2 =\sum_{K \in \mathcal T_h}\int_K[\mu\,\nabla\vv:\nabla\vv+(\mu+\lambda) (\nabla \cdot \vv)^2] \,d\bsx\,.
\]
There exists a positive constant $B$ (depending on $\Omega$) such that 
\begin{equation}\label{eq: B bound}
    \|\vv\| \le  B\,\|\vv\|_h \le  B/\sqrt{\mu_{\min}} \|\vv\|_{\calB_h},\quad {\rm for~all}~~\vv \in \vV_h;
\end{equation}
 see \cite{Brenner2003} for the proof of the first inequality. Noting that  $\|\vv\|_h=0$  means that $\vv$ is a piecewise constant vector on each $K \in \mathcal T_h$, and hence, by using the zero boundary conditions together with continuity at midpoints, it follows that $\vv\equiv {\bf 0}.$

Assume that  $\max\{|\mu_{x_1}|,|\mu_{x_2}|\}\le c_0$ for some positive constant $c_0$. By using \eqref{eq: B bound}, we get
\begin{multline*}
    2\sum_{K \in \mathcal T_h}\int_K|\nabla \mu\cdot \vk_{v_1} v_2|\,d\bsx \le 
2c_0\sum_{K \in \mathcal T_h}\int_K(|v_{1_{x_1}}| +|v_{1_{x_2}}|)|v_2|\,d\bsx \\
\le 
c_0\sum_{K \in \mathcal T_h}\int_K(|v_{1_{x_1}}|^2 +|v_{1_{x_2}}|^2+2|v_2|^2)\,d\bsx\\
\le c_0\Big(\|\vv\|_h^2+2\|\vv\|^2\Big)
\le 
c_0(1+2B^2)/\mu_{\min}\,\|\vv\|_{\calB_h}^2\,.
\end{multline*}
Consequently, for $\vv \in \vV_h$,
  \begin{equation}\label{Bh coer} 
(1- c_0(1+2B^2)/\mu_{\min})\,\|\vv\|_{\calB_h}^2\le    \widehat\calB_{h}(\bsr;\vv,\vv) \le 
(1+ c_0(1+2B^2)/\mu_{\min})\,\|\vv\|_{\calB_h}^2\,.\end{equation}
Therefore, $\widehat\calB_{h}(\bsr;\cdot,\cdot)$ is positive-definite over  $\vV_h$ whenever
$c_0<\mu_{\min}/(1+2B^2)$. Noting that $\widehat\calB_h(\bsr;\cdot,\cdot)=\widehat\calB(\bsr;\cdot,\cdot)$ on $\vV$, we can show with the help of the Poincar\'e inequality that $\widehat\calB(\bsr;\cdot,\cdot)$ is also positive-definite on $\vV$ for small $c_0$, and consequently, $\widehat\calB_{h}(\bsr;\cdot,\cdot)$ is positive-definite on $\vV_h+\vV$ for small $c_0$. The associated energy norm is $\|\cdot\|_{\widehat\calB_h}=\Big(\widehat\calB_h(\bsr;\cdot,\cdot)\Big)^{1/2}.$
\subsection{Convergence analysis}
For the error analysis, following \cite{CrouzeixRaviart1973}, we use the projection operator $\Pi_h: \vV \to \vV_h$,  defined by  
\[\int_{E_i} [\Pi_h \vv-\vv]\,ds=0 \iff \Pi_h \vv(\xi_i)=\frac{1}{|E_i|}\int_{E_i} \vv\,ds,\quad{\rm for}~~1\le i\le q_h.\]
Thus, for any $K \in \mathcal T_h$, $\nabla_h \cdot (\Pi_h \vv) =|K|^{-1}\int_K \nabla \cdot \vv\,d\bsx$ on $K$ , see \cite{CrouzeixRaviart1973} for the proof. So, 
\[\nabla \cdot \vv=0 \implies  \nabla_h \cdot (\Pi_h \vv)=0.\] 
By a standard Bramble–Hilbert argument, it is known that  
\begin{equation}
    \label{projection estimate}\|\vv-\Pi_h \vv\|+h\|\nabla_h (\vv-\Pi_h \vv)\| \le C h^2 \|\vv\|_{\bf H}\,.
\end{equation}
 
In the next theorem, we use some ideas from \cite{BrennerSung1992,CrouzeixRaviart1973} to show the convergence  of the nonconforming finite element solution of problem \eqref{eq:L1} over the domain $\Omega$.   
\begin{theorem}\label{Convergence theorem}
Assume that $\Omega$ is convex,  $\vf \in {\bf L}^2(\Omega)$, and $\max\{|\mu_{x_1}|,|\mu_{x_2}|\}\le c_0$ on $\Omega\times U$ for a sufficiently small  $c_0$. For every $\bsr  \in U\times U,$ let $\vu$ and $\vu_h$ be the solutions of problems \eqref{eq:L1} and \eqref{FEM new}, respectively.  Under Assumption \eqref{ass A5}, we have $\|\vu(\cdot, \bsr)-\vu_h(\cdot, \bsr)\|_{\widehat\calB_h}  
\le  C  h \|\vf\|$,  where $C$ depends on $\Omega$,   $\mu$ and  $\widehat \lambda$,  and linearly on   $\|\nabla \mu(\cdot,\bsy)\|_{L^\infty(\Omega)}$ and $\|\nabla \widehat \lambda(\cdot,\bsz)\|_{L^\infty(\Omega)}$, but is independent of the Lam\'e parameter factor $\Lambda$.
\end{theorem}
\begin{proof} 
The bilinear operator $\widehat\calB_h(\bsr;\cdot,\cdot)$ is symmetric and positive-definite on $\vV_h+\vV$,  and so, by the second Strang lemma \cite{StrangFix1973},  we have  
\begin{equation}\label{Strang}
    \|\vu-\vu_h\|_{\widehat\calB_h} \le \inf_{\vv_h \in \vV_h}\|\vu-\vv_h\|_{\widehat\calB_h}+ 
\sup_{\vv_h \in \vV_h \backslash\{0\}}\frac{|\widehat \calB_h(\bsr;\vu-\vu_h,\vv_h)|}{\|\vv_h\|_{\widehat\calB_h}}\,.
\end{equation}
Using \eqref{Bh coer}, \eqref{u*}, the identity $\nabla_h \cdot \Pi_h \vu^*=\nabla_h \cdot \Pi_h \vu$ (because $\nabla_h \cdot \vu^*=\nabla_h \cdot \vu$),  the projection bound in \eqref{projection estimate}, and the second regularity estimate in  Theorem \ref{lem: vu bound}, we reach 
\begin{equation}\label{Strang1}
\begin{aligned}
    \inf_{\vv_h \in \vV_h}\|\vu-&\vv_h\|_{\widehat\calB_h}^2 \le C\|\vu-\Pi_h \vu\|_{\widehat\calB_h}^2\\
    &\le 
C \|\vu-\Pi_h \vu\|_h+C \Lambda \|\nabla_h\cdot(\vu^*-\Pi_h \vu^*)\|+C \|\vu-\Pi_h \vu\|_h \|\vu-\Pi_h \vu\|\\
&\le 
Ch^2\Big(\|\vu\|_{\bf H}+\Lambda \|\vu^*\|_{\bf H}\Big)^2
\le 
Ch^2\Big(\|\vu\|_{\bf H}+\Lambda \|\nabla \cdot \vu\|_\vV\Big)^2 \le C h^2\|\vf\|^2\,.
\end{aligned}
\end{equation}
To estimate the second term in \eqref{Strang}, we make the splitting
\begin{equation}\label{J123}
    \widehat \calB_h(\bsr;\vu-\vu_h,\vv_h)=\widehat \calB_h(\bsr;\vu,\vv_h)-\iprod{\vf,\vv_h}=J_1 +J_2 -J_3,
\end{equation}
where, after
applying the divergence theorem over each element,  \begin{align*}
    J_1&=\sum_{K \in \mathcal T_h}\int_{\partial K} \mu\, (\nabla \vu\,\vn)\cdot \vv_h\,d\bsx, \quad 
    J_2=\sum_{K \in \mathcal T_h}\int_{\partial K} 
(\mu+\lambda)(\nabla \cdot \vu)\,\vn \cdot \vv_h\,d\bsx,\\
&{\rm and}~~J_3=\sum_{K \in \mathcal T_h}\int_{\partial K} v_{1_h} u_2\,\vk_\mu\cdot \vn \,d\bsx,
\end{align*}
with $\vu=(u_1,u_2)^T$ and $\vv_h=(v_{1_h},v_{2_h})^T$, and with $\vn$ denoting the outward unit normal vector. 

The task now is to estimate $J_1$, $J_2$ and $J_3.$ For $J_1,$   define the jump matrix function   $\lb\vv_h\rb $  across the edge $e \in \calE$ between the two elements $K_1$ and $K_2$, with outward unit normal vectors $\vn_1$ and $\vn_2$, respectively, as $\lb\vv_h\rb=\vv_h|_{K_1}\vn_1^T+\vv_h|_{K_2}\vn_2^T$. If $e \in \calF$ is a boundary edge that belongs to $K_1$, then $\lb\vv_h\rb=\vv_h|_{K_1}\vn_1^T$. Thus, $J_1$ can be expanded  as 
\[J_1=\sum_{e \in \calE\cup \calF}\int_e \mu\, \nabla \vu: \lb\vv_h\rb\,d\bsx\,.\]
Due to the continuity of the function $\vv_h$ at the midpoint of any edge  $e \in \calE$ and since $\vv_h={\bf 0}$ at the midpoint of any edge $e \in  \calF$, $\lb\vv_h\rb={\bf 0}$ on the midpoint of every $e \in \calE\cup \calF$.  Consequently, by using the Bramble-Hilbert Lemma and homogeneity arguments, and with the help of the Trace Theorem 
(see \cite[Section 1.6]{BrennerScott2008} or \cite{Scott1977}),
we get 
\[J_1\le Ch \Big(\sum_{e \in \calE\cup \calF}\|\mu\nabla \vu\|_{H^1(\bar K_e)}\|^2\Big)^{1/2} 
 \Big(\sum_{e \in \calE\cup \calF}|e|\lb\vv_h\rb\|_{{\bf L}^2(e)}^2\Big)^{1/2}.\]
Here, $\bar K_e \in K_e$ where $K_e$ is the set of triangles in ${\mathcal T}_h$ having $e$ as an edge. Hence, using the inequality $|e|\,\|\lb\vv_h\rb\|^2_{{\bf L}^2(e)} \le C h^2\sum_{K \in K_e}  \|\nabla \vv_h\|^2_{{\bf L}^2(K)}$ (because $\lb\vv_h\rb={\bf 0}$ at the midpoint of $e$) yields 
\[|J_1| \le  C h\|\mu\nabla \vu\|_\vV \|\vv_h\|_{\calB_h}\le C h\|\vu\|_{\bf H} \|\vv_h\|_{\calB_h}\,.\]
To estimate $J_2$, we follow the above procedure,  taking into account that the jump scalar function     $\lb\vv_h\rb=\vv_h|_{K_1}\cdot \vn_1+\vv_h|_{K_2}\cdot \vn_2$ across the edge $e \in \calE$, while $\lb\vv_h\rb=\vv_h|_{K_1}\cdot \vn_1$ on the boundary edge $e$ that belongs to $K_1$. We get $|J_2| \le   C h\|(\mu+ \lambda)\nabla\cdot\vu\|_\vV \|\vv_h\|_{\calB_h}.$
Again, a similar procedure can also be followed to estimate $J_3$ and obtain  
\[|J_3| \le   C h\|\nabla \mu\|_{W^{1,\infty}(\Omega)}\|u_2\|_\vV \|v_{1_h}\|_{\calB_h}\le   C h\|\vu\|_\vV \|\vv_h\|_{\calB_h}\,,\]
where  $\lb\vv_h\rb=v_{1_h}|_{K_1} \vn_1+v_{1_h}|_{K_2} \vn_2$ on $e \in \calE\cap \overline K_1\cap \overline K_2$ and  $\lb\vv_h\rb=v_{1_h}|_{K_1} \vn_1$ on  $e \in \calF\cap \overline K_1$.
Now, inserting the above estimates in \eqref{J123} and using  Theorem \ref{lem: vu bound} and \eqref{Bh coer} yield  
\begin{equation}\label{J123 estimate}
    |\widehat \calB_h(\bsr;\vu-\vu_h,\vv_h)|\le Ch\Big(\|\vu\|_{\bf H} +\Lambda\|\nabla \cdot\vu\|_{\vV}\Big) \|\vv_h\|_{\calB_h} \le Ch\|\vf\|\,\|\vv_h\|_{\widehat \calB_h}\,.
\end{equation}
Using this and  \eqref{Strang1} together with \eqref{Strang} will complete the proof of this theorem. 
\end{proof}
 For measuring the  QMC finite element error in \eqref{combine}, we estimate $\calL(\vu-\vu_h)$ in the next theorem. 
\begin{theorem}\label{Convergence theorem QMC}
In addition to the  assumptions of Theorem \ref{Convergence theorem},  we assume that  $\calL:{\bf L}^2(\Omega) \to \R$ is a bounded  linear functional ($|\calL(\vw)|\le \|\calL\|\,\|\vw\|$), then 
\begin{equation*}
|\calL(\vu(\cdot, \bsr))-\calL(\vu_h(\cdot, \bsr))|
\le C h^2 \|\vf\| \|\calL\|,~~~{\rm for~every}~~\bsr=(\bsy,\bsz)  \in U\times U,
\end{equation*}
 where $C$ depends on $\Omega$,  $\mu$, $\widehat \lambda$,  and linearly on    $\|\nabla \mu(\cdot,\bsy)\|_{L^\infty(\Omega)}$ and $\|\nabla \widehat \lambda(\cdot,\bsz)\|_{L^\infty(\Omega)}$, but is independent of the Lam\'e parameter factor $\Lambda$.  Furthermore,  by using  the equality 
\[\|\vv\|= \sup_{\vw\in {\bf L}^2(\Omega),\vw\ne {\bf 0}} \frac{|\calL(\vv)|}{\|\vw\|},\quad{\rm for ~any}~~\vv \in {\bf L}^2(\Omega),~~{\rm with}~~\calL(\vv)=\iprod{\vw,\vv},\]
the above estimate leads to the following  optimal convergence bound:
\[\|\vu(\cdot, \bsr)-\vu_h(\cdot, \bsr)\|\le C\,h^2 \|\vf\|,~~~{\rm for~every}~~\bsr  \in U\times U\,.\]
\end{theorem}
\begin{proof} 
We start our proof by replacing $\ell$ with $\calL$ in \eqref{para weak}, and then, by Theorem \ref{thm: unique solution}, the new parametric variational problem: find $\vu_{\calL} \in \vV$ such that   
\begin{equation}\label{weak form uL}
    \calB(\bsr;\vu_{\calL}, \vv) = \calL(\vv)\quad{\rm  for~all}~~  \vv \in \vV,~~{\rm for~every}~~\bsr \in U\times U,
\end{equation}
has a unique solution. Furthermore, using  Theorem \ref{lem: vu bound} (with $\vu_{\calL}$ in place of $\vu$), we deduce that  \begin{equation}\label{regulairty of u_L}
    \|\vu_{\calL}(\cdot,\bsr)\|_{\vV}+\Lambda\|\nabla\cdot \vu_{\calL}(\cdot,\bsr)\| \le C\|\calL\|_{\vV^*}\,.
\end{equation} 
Then by proceeding as in the proof of \eqref{a priori H2} in Theorem \ref{lem: vu bound} (with $\vu_{\calL}$ in place of $\vu$), and using that the linear functional $\calL$ is bounded in ${\bf L}^2(\Omega)$, we conclude that   \begin{equation}\label{regularity of ucalL}
\|\vu_{\calL}(\cdot,\bsr)\|_{\bf H} + \|(\lambda+\mu)\nabla\cdot \vu_{\calL}(\cdot,\bsr)\|_{\vV}\le C\|\calL\|.
\end{equation}
The nonconforming finite element solution of $\vu_{\calL}$ is defined by: Find ${\bf\Theta}_h \in \vV_h$ such that 
\begin{equation}\label{weak form uL discrete}
    \widehat  \calB_h(\bsr; {\bf \Theta}_h, \vv_h) = \calL(\vv_h), \quad \text{for all} \quad \vv_h \in  \vV_h\,.
\end{equation}
By following the proof of Theorem \ref{Convergence theorem} and using  \eqref{regularity of ucalL}, we obtain the following error estimate   
\begin{equation}\label{projection estimate ucalL}
    \|\vu_{\calL}(\cdot,\bsr)-{\bf\Theta}_h(\cdot,\bsr)\|_{\widehat \calB_h}\le Ch\|\vu_{\calL}(\cdot,\bsr)\|_{\bf H}\le Ch\|\calL\|.
\end{equation}
From the linearity property of $\calL$, \eqref{weak form uL}, \eqref{weak form uL discrete}, and the symmetric property of  $\widehat \calB_h(\bsr;\cdot,\cdot)$, 
\begin{equation}\label{G1G2G3}
    \calL(\vu(\cdot,\bsr)-\vu_h(\cdot,\bsr))
=\widehat  \calB_h(\bsr;\vu_{\calL}, \vu)
-\widehat  \calB_h(\bsr;{\bf\Theta}_h, \vu_h)=G_1+G_2+G_3,
\end{equation}
where 
\begin{multline*}
G_1=\widehat  \calB_h(\bsr;\vu_{\calL}-{\bf\Theta}_h, \vu-\Pi_h \vu)
+\widehat\calB_h(\bsr;\vu
- \vu_h,{\bf\Theta}_h-\vu_{\calL})
+\widehat\calB_h(\bsr;\vu
- \vu_h,\vu_{\calL}-\Pi_h \vu_{\calL}),
\end{multline*}
\[
G_2=\widehat  \calB_h(\bsr;\vu
- \vu_h,\Pi_h \vu_{\calL})\quad{\rm and}\quad G_3=\widehat  \calB_h(\bsr;\vu_{\calL}-{\bf\Theta}_h,\Pi_h \vu)\,.\]
By the Cauchy-Schwarz inequality, Theorem  \ref{Convergence theorem},   \eqref{projection estimate ucalL}, the bound in \eqref{Strang1} in addition to a similar bound with $\vu_{\calL}$ in place of $\vu$ where the regularity estimate in \eqref{regularity of ucalL} is needed,  we reach 
\begin{multline*}
|G_1|
\le 
\|\vu_{\calL}-{\bf\Theta}_h\|_{\widehat \calB_h}\| \vu-\Pi_h\vu\|_{\widehat \calB_h}+\|\vu- \vu_h\|_{\widehat \calB_h}(\|{\bf\Theta}_h-\vu_{\calL}\|_{\widehat \calB_h}+\|\vu_{\calL}-\Pi_h\vu_{\calL}\|_{\widehat \calB_h}) \le Ch^2 \|\vf\|\,\|\calL\|. \end{multline*}
To estimate $G_2$, we follow the approach used for bounding \eqref{J123}, that is, the derivation between \eqref{J123} and \eqref{J123 estimate}. To obtain a sharp bound, it is necessary  to use the identity  $\lb\Pi_h \vu_{\calL}\rb=\lb\Pi_h \vu_{\calL}-\vu_{\calL}\rb$ across the skeleton of the finite element mesh, which is due to the continuity of $\vu_{\calL}$ on $\Omega$ and also to the fact that $\vu_{\calL}(\cdot,\bsr)={\bf 0}$ on $\Gamma$. In conclusion,  we get 
\[|G_2|\le Ch\Big(\|\vu\|_{\bf H} +\Lambda\|\nabla \cdot\vu\|_{\vV}\Big) \|\Pi_h\vu_{\calL}-\vu_{\calL}\|_h \le Ch^2\|\vf\|\,\|\calL\|\,.
\]
By  arguing as above,  we obtain a similar  estimate of $G_3$ with  $\vu_{\calL}$ and $\vu$ in place of $\vu$ and $\vu_{\calL}$, respectively, and so,  $|G_3|\le Ch^2\|\vf\|\,\|\calL\|.$  Finally,  inserting the achieved estimates of $G_1$, $G_2$ and $G_3$ in \eqref{G1G2G3} will complete the proof. 
\end{proof}


\section{Numerical experiments}\label{Sec: Numeric}
In this section, we illustrate the theoretical convergence results using four examples. The source code is hosted on GitHub~\cite{gitrepo}.

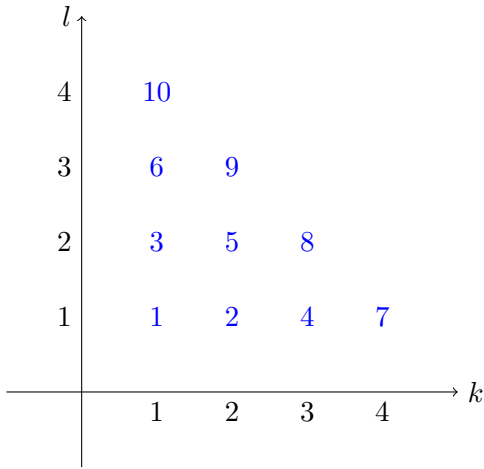
\begin{figure}
\caption{The numbers in blue give the values of $j=\rho(k,l)$ if $n=4$.}
\label{fig: rho}
\begin{center}
\begin{tikzpicture}
\draw[->] (-1,0) -- (5,0);
\node[right] at (5,0) {$k$};
\draw[->] (0,-1) -- (0,5);
\node[left] at (0,5) {$l$};
\foreach \k in {1,2,3,4}
    \node[below] at (\k,0) {$\k$};
\foreach \l in {1,2,3,4}
    \node[left] at (0,\l) {$\l$};
\node[blue] at (1,1) {$1$};
\node[blue] at (2,1) {$2$};
\node[blue] at (1,2) {$3$};
\node[blue] at (3,1) {$4$};
\node[blue] at (2,2) {$5$};
\node[blue] at (1,3) {$6$};
\node[blue] at (4,1) {$7$};
\node[blue] at (3,2) {$8$};
\node[blue] at (2,3) {$9$};
\node[blue] at (1,4) {$10$};
\end{tikzpicture}
\end{center}
\end{figure}

\begin{table}
\caption{Errors and convergence rates for the problem in Example~1  when $\Lambda=1$.}\label{tab: Lambda=1}
\begin{center}
{\tt
\begin{tabular}{|cr|cc|cc|}
\multicolumn{6}{c}{\textrm{Conforming elements}}\\
\hline
$h$&\multicolumn{1}{c|}{\textrm{DoF}}&\multicolumn{2}{c|}{\textrm{$\bfL^2$ error}}&
\multicolumn{2}{c|}{\textrm{$\bfH^1$ error}}\\
\hline
 0.262&    552& 1.22e-01&      & 1.09e+00&      \\
 0.131&   2330& 3.11e-02& 1.976& 5.43e-01& 1.008\\
 0.065&   9570& 7.82e-03& 1.992& 2.71e-01& 1.002\\
 0.033&  38786& 1.96e-03& 1.998& 1.36e-01& 1.001\\
 0.016& 156162& 4.90e-04& 1.999& 6.78e-02& 1.000\\
\hline
\multicolumn{6}{c}{}\\
\multicolumn{6}{c}{\textrm{Nonconforming elements}}\\
\hline
$h$&\multicolumn{1}{c|}{\textrm{DoF}}&\multicolumn{2}{c|}{\textrm{$\bfL^2$ error}}&
\multicolumn{2}{c|}{\textrm{$\bfH^1$ error}}\\
\hline
 0.262&   1778& 4.20e-02&      & 1.07e+00&      \\
 0.131&   7240& 1.06e-02& 1.992& 5.34e-01& 0.997\\
 0.065&  29216& 2.64e-03& 1.997& 2.67e-01& 0.999\\
 0.033& 117376& 6.62e-04& 1.999& 1.34e-01& 1.000\\
 0.016& 470528& 1.65e-04& 2.000& 6.68e-02& 1.000\\
\hline
\end{tabular}
}
\end{center}
\end{table}

\begin{table}
\caption{Errors and convergence rates for the problem in Example~1  when $\Lambda=1000$.}\label{tab: Lambda=1000}
\begin{center}
{\tt
\begin{tabular}{|cr|cc|cc|}
\multicolumn{6}{c}{\textrm{Conforming elements}}\\
\hline
$h$&\multicolumn{1}{c|}{\textrm{DoF}}&\multicolumn{2}{c|}{\textrm{$\bfL^2$ error}}&
\multicolumn{2}{c|}{\textrm{$\bfH^1$ error}}\\
\hline
 0.262&    552& 1.48e+00&      & 3.65e+00&      \\
 0.131&   2330& 5.80e-01& 1.354& 1.62e+00& 1.172\\
 0.065&   9570& 1.81e-01& 1.681& 6.34e-01& 1.355\\
 0.033&  38786& 5.04e-02& 1.842& 2.47e-01& 1.361\\
 0.016& 156162& 1.34e-02& 1.915& 9.86e-02& 1.323\\
\hline
\multicolumn{6}{c}{}\\
\multicolumn{6}{c}{\textrm{Nonconforming elements}}\\
\hline
$h$&\multicolumn{1}{c|}{\textrm{DoF}}&\multicolumn{2}{c|}{\textrm{$\bfL^2$ error}}&
\multicolumn{2}{c|}{\textrm{$\bfH^1$ error}}\\
\hline
 0.262&   1778& 4.01e-02&      & 1.01e+00&      \\
 0.131&   7240& 1.01e-02& 1.986& 5.09e-01& 0.997\\
 0.065&  29216& 2.54e-03& 1.995& 2.54e-01& 0.999\\
 0.033& 117376& 6.35e-04& 1.998& 1.27e-01& 1.000\\
 0.016& 470528& 1.59e-04& 1.999& 6.36e-02& 1.000\\
\hline
\end{tabular}
}
\end{center}
\end{table}

\paragraph{Example 1.} 
We consider a purely deterministic problem with Lam\'e parameters $\mu=\mu_0$~and $\lambda=\lambda_0$, in effect taking the random variables $\bsy=\bsz=0$ in the expansions~\eqref{KLexpansion}. Based on Theorems \ref{Convergence theorem} and \ref{Convergence theorem QMC}, we expect to observe optimal order convergence rates in both ${\bf H}^1(\Omega)$ and  ${\bf L}^2(\Omega)$-norms, uniformly in $\Lambda$, provided that all the imposed assumptions are  satisfied.  We chose $\lambda_0(\boldsymbol{x}) = \Lambda(1+\tfrac12\sin(2x_1))$ and $\mu_0(\boldsymbol{x}) = 1 + x_1+x_2$  for $\bsx=(x_1,x_2)$ in the square $\Omega=(0,\pi)^2$, and chose the body force so that the exact solution is
\[\vu(\boldsymbol{x})=\begin{bmatrix}  u_1(x_1,x_2)\\  u_2(x_1,x_2)\end{bmatrix}
=\begin{bmatrix} \bigl(\cos(2x_1)-1\bigr)\sin(2x_2)\\ \bigl(1-\cos(2x_2)\bigr)\sin(2x_1)
\end{bmatrix} +\frac{\sin(x_1)\sin(x_2)}{\Lambda}\begin{bmatrix}1\\ 1\end{bmatrix}.\] 
Notice $\vu=\mathbf{0}$ on~$\partial\Omega$.  Table~\ref{tab: Lambda=1} compares the errors and convergence rates of our nonconforming scheme with those using the standard conforming, piecewise-linear Galerkin finite element method, when~$\Lambda=1$.  We used a family of unstructured triangulations constructed by uniform refinement of an initial coarse mesh.  As expected, we observe convergence of order $h^2$ in the $\bfL^2$ norm and order~$h$ in the (broken) $\bfH^1$ norm. For a given mesh size~$h$, the nonconforming scheme requires more than three times as many degrees of freedom as the conforming method, but with an $\bfL^2$ error about a third as large.  The errors in the (broken) $\bfH^1$ norm  are essentially the same for both methods.  However, when $\Lambda=1000$ (a nearly incompressible case), as expected,  we see in Table~\ref{tab: Lambda=1000} that the accuracy of the conforming method is much poorer, whereas the errors for the nonconforming scheme are almost unchanged.

\begin{figure}
\caption{Instance of $\lambda$ from Example~2 with $\alpha=2$, $\Lambda=1$
and $s_1=253$.}\label{fig: random coef}
\begin{center}
\includegraphics[scale=0.6]{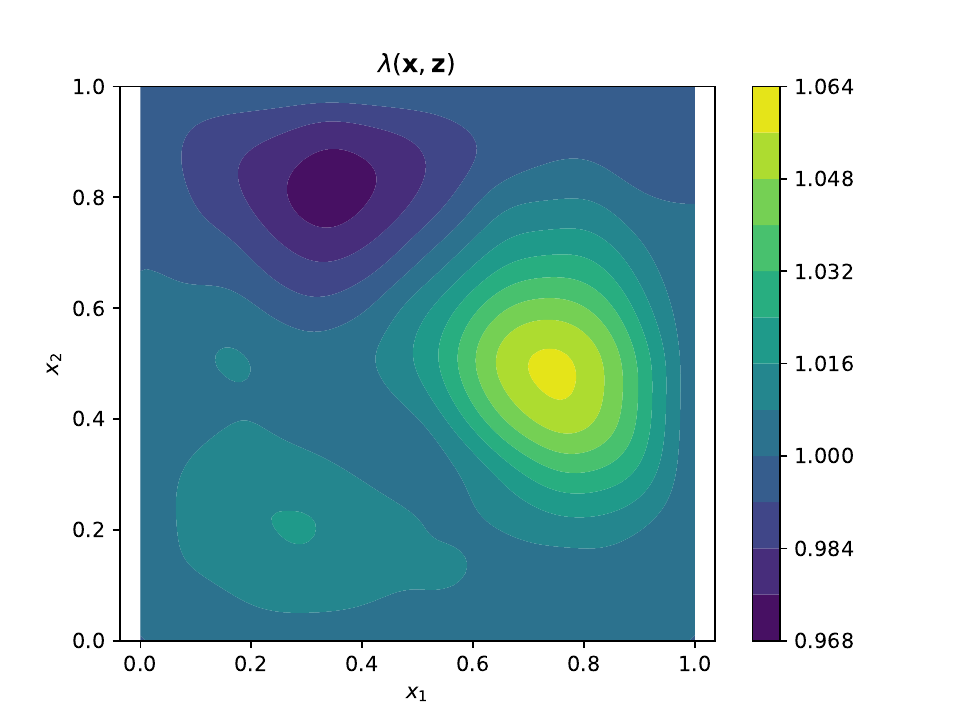}    
\end{center}
\end{figure}

For the remaining examples, we chose the \emph{unit} square~$\Omega=(0,1)^2$ as our spatial domain. To construct random coefficients, let $\alpha>1$ and   define
\[
\psi_{kl}(\bsx)=\phi_{kl}(\bsx)=\frac{\sin(k\pi x_1)\,\sin(l\pi x_2)}{M_\alpha(k+l)^{2\alpha}}
\quad\text{where}\quad
M_\alpha=\lim_{n\to\infty}\sum_{l=1}^n\sum_{k=1}^{n-l}\frac{1}{(k+l)^{2\alpha}}
    =\sum_{j=1}^\infty\frac{j-1}{j^{2\alpha}}.
\]
Notice $M_\alpha=\zeta(2\alpha-1)-\zeta(2\alpha)$ in terms of the Riemann zeta function. By defining $\rho(k,l)$ as shown in Figure~\ref{fig: rho}, we can switch to single-index labelling, putting 
\[
z_j=z_{kl}\quad\text{and}\quad\psi_j=\phi_j=\psi_{kl}\quad\text{for $j=\rho(k,l)$.}
\]
We find that
\[
1+\smash[t]{\sum_{i=1}^{m-2}i\le\rho(k,l)\le\sum_{i=1}^{m-1}i=\frac{m(m-1)}{2}}
    \quad\text{if $k+l=m$,}
\]
and in particular if $j=\rho(k,l)$ then $2j\le(k+l)^2$ so
\[
\|\psi_j\|_{L^\infty(\Omega)}\le\smash[b]{\frac{1}{M_\alpha(k+l)^{2\alpha}}
    \le\frac{(2j)^{-\alpha}}{M_\alpha}}
\]
and
\[
\|\nabla\psi_j\|_{L^\infty(\Omega)}
    \le\smash[t]{\frac{\pi(k^2+l^2)^{1/2}}{M_\alpha(k+l)^{2\alpha}}
    \le\frac{\pi(k+l)^{1-2\alpha}}{M_\alpha}
    \le\frac{\pi(2j)^{-(\alpha-\frac12)}}{M_\alpha}.}
\]
Consider the sum
\[
S_{n,\alpha}(\bsx)=\smash[t]{\sum_{l=1}^n\sum_{k=1}^{n-l}z_{kl}\psi_{kl}(\bsx)
    =\sum_{k=1}^M\sum_{l=1}^Ma_{kl}\sin(k\pi x_1)\sin(l\pi x_2)},
\]
where $a_{kl}=M_\alpha^{-1}(k+l)^{-2\alpha}$ if $1\le k\le n-l$ and $1\le l\le n$, and zero otherwise. Since $|z_{kl}|\le1/2$,  $-1/2\le S_{n,\alpha}(\bsx)\le1/2$ for all $n$~and $\bsx$, so our assumptions (A1)--(A3) hold for $\alpha>3/2$, and (A4) holds for $1/\alpha<p<1$ (and $q=p$).  In all the computations described below,  $\alpha=2$.

\begin{table}
\caption{Errors and convergence rates for $\I_N\calL_h$
in Example~2 for $\Lambda=1$ (left) and $\Lambda=1,000$ (right).
Thus, $\mu=\mu_0$ is deterministic and $\lambda$ is random.}\label{tab: random lambda}
\begin{center}
\renewcommand{\arraystretch}{1.2}
{\tt
\begin{tabular}{|r|ccc|ccc|}
\multicolumn{4}{c}{$\Lambda=1$}&\multicolumn{3}{c}{$\Lambda=1000$}\\
\hline
$N$&$\I_N\calL_h$&Error& Rate&$\I_N\calL_h$&Error& Rate\\
\hline
    16&  -0.4024722078&  1.74e-05&         &  -0.0031572258&  9.85e-07&         \\
    32&  -0.4024592212&  4.38e-06&    1.986&  -0.0031565029&  2.62e-07&    1.912\\
    64&  -0.4024559125&  1.07e-06&    2.028&  -0.0031563107&  6.95e-08&    1.913\\
   128&  -0.4024551004&  2.63e-07&    2.032&  -0.0031562566&  1.54e-08&    2.173\\
\hline
\end{tabular}
}
\end{center}    
\end{table}

By applying a 2-dimensional fast sine transform to the array~$a_{kl}$, we obtained the values of $S_{n,\alpha}$ on an $M\times M$ grid over the unit square~$\Omega$.  Bivariate cubic spline interpolation was then used to compute $S_{n,\alpha}(\bsx)$ at each quadrature point~$\bsx$ required by the finite element solver. We chose $M=256$ for computing $\lambda$~and $\mu$, and $M=512$ for computing $\nabla\mu$ (the latter requiring mixed fast sine and cosine transforms).  Code profiling showed that for each $\bsr=(\bsy, \bsz)$, the CPU time to compute $\vu_h(\cdot,\bsr)$ was split roughly as follows: 40\% for constructing the interpolants, 40\% for assembling the linear system, and 20\% for solving the linear system.  We also observed that naive evaluation of $S_{n,\alpha}$ was about 20 times slower than using fast transforms and interpolation.  

\begin{table}
\caption{Errors and convergence rates for the expected value of 
$\calL_h$ in Example~3 for $\Lambda=1$ (left) and
$\Lambda=1000$ (right). Thus, $\lambda=\lambda_0$ is deterministic and $\mu$ is random.}
\label{tab: random mu}
\begin{center}
\renewcommand{\arraystretch}{1.2}
{\tt
\begin{tabular}{|r|ccc|ccc|}
\multicolumn{4}{c}{$\Lambda=1$}&\multicolumn{3}{c}{$\Lambda=1000$}\\
\hline
$N$&$\I_N\calL_h$&Error& Rate&$\I_N\calL_h$&Error& Rate\\
\hline
    16&  -0.6984920721&  1.73e-04&         &  -0.0036268972&  1.07e-08&         \\
    32&  -0.6983612152&  4.21e-05&    2.038&  -0.0036268891&  2.67e-09&    2.005\\
    64&  -0.6983301483&  1.11e-05&    1.930&  -0.0036268871&  6.60e-10&    2.015\\
   128&  -0.6983216359&  2.54e-06&    2.122&  -0.0036268866&  1.59e-10&    2.057\\
\hline
\end{tabular}
}
\end{center}    
\end{table}

For the linear functional, we chose the average displacement in the $x_2$-direction, that is, $\mathcal{L}(\vu)=\int_\Omega u_2\,d\bsx$, and for the source term, $\vf(\bsx)=[1-x_2^2,2x_1-20]$.  Starting from an unstructured triangular mesh on~$\Omega$, we performed three uniform refinements to obtain a family of four meshes, the finest of which had $h=0.0190$ and $30,848$ free nodes.   For each $\bsr$ and each mesh, we computed the finite element solution~$\vu_h(\cdot,\bsr)$ and evaluated $\calL(\vu_h(\cdot,\bsr))=\calL(\vu(\cdot,\bsr))+O(h^2)$.  Richardson extrapolation then yielded a value~$\mathcal{L}_h(\bsr)=\mathcal{L}(\vu(\cdot,\bsr))+O(h^8)$. The finite element equations were solved via the conjugate gradient method, preconditioned by the stiffness matrix~$\mathbf{P}$ obtained from using constant coefficients $\lambda_0(\bsx)=\Lambda$~and $\mu_0(\bsx)=1$.  In this way, after computing the sparse Cholesky factorization of~$\mathbf{P}$ once for a given mesh, only around ten iterations were needed to compute $\vu_h(\cdot,\bsr)$ for each choice of~$\bsr$.

\paragraph{Example 2.} 
Consider a deterministic coefficient~$\mu(\bsx)=\mu_0(\bsx)=1+x_1+x_2$, together with a random coefficient $\lambda$ defined by taking $\widehat\lambda_0(\bsx)=1$ in~\eqref{KLexpansion}.  In our computations, we worked with the truncated expansion $\lambda_{s_2}(\bsx,\bsz)=\Lambda[1+S_{n,\alpha}(\bsx)]$ for $n=22$, so that $s_2=n(n+1)/2=253$, and with the values $\Lambda=1$~and $\Lambda=1000$. Figure~\ref{fig: random coef} shows the contour plot of $\lambda_{s_2}$ for a particular choice of~$\bsz$. Keeping $s_2$~and $h$ fixed, Table~\ref{tab: random lambda} shows the behaviour of the error in the expected value $\I_N\calL_h$ as we increase the number~$N$ ($N=N_2$)  of the high-order QMC points (generated by a Python package in \cite{Gantner2014}). The error was calculated by comparison to the reference value obtained with $N=512$. Consistent with Theorem \ref{sec: main results}  (or Theorem~\ref{prop:qmc6.3}) for $q>1/\alpha=1/2$, we observe convergence of order~$N^{-1/q}\approx N^{-2}$ for both choices of~$\Lambda$.

\paragraph{Example~3.}  Next, consider a deterministic $\lambda(\bsx)=\lambda_0(\bsx)=\Lambda[1+\tfrac12\sin(2\pi x_1)]$ and a random $\mu$, putting $\mu_0(\bsx)=1$ in \eqref{KLexpansion}.  The finite element and QMC approximations were performed as in Example~2, with $s_1=253$. In Table~\ref{tab: random mu}, we once again observe convergence of order~$N^{-2}$ ($N=N_1$), both for $\Lambda=1$ and $\Lambda=1000$. The error was calculated by comparison to the reference value obtained with $N=512$. Again, consistent with Theorem \ref{sec: main results}  (or Theorem~\ref{prop:qmc6.3}) for $p>1/\alpha=1/2$, we observe a convergence of order~$N^{-1/p}\approx N^{-2}$ for both choices of~$\Lambda$. 

\begin{table}
\caption{Errors and convergence rates for the expected value of 
$\calL_h$ in Example~4 for $\Lambda=1$ (left) and
$\Lambda=1000$ (right). Thus, both $\lambda$ and $\mu$ are random.}
\label{tab: random lambda mu}
\begin{center}
\renewcommand{\arraystretch}{1.2}
{\tt
\begin{tabular}{|r|ccc|ccc|}
\multicolumn{4}{c}{$\Lambda=1$}&\multicolumn{3}{c}{$\Lambda=1000$}\\
\hline
$N$&$\I_N\calL_h$&Error& Rate&$\I_N\calL_h$&Error& Rate\\
\hline
    16&  -0.6942608771&  3.39e-04&         & -0.0031639526&  1.01e-06&         \\
    32&  -0.6941636764&  2.42e-04&    0.488& -0.0031632154&  2.71e-07&    1.896\\
    64&  -0.6939439764&  2.19e-05&    3.465& -0.0031630179&  7.34e-08&    1.884\\
   128&  -0.6939282911&  6.20e-06&    1.820& -0.0031629609&  1.64e-08&    2.161\\
\hline
\end{tabular}
}
\end{center}    
\end{table}

\paragraph{Example~4.} Finally, consider both $\lambda$~and $\mu$ random.  This time, we
chose $n=15$ so that $s_1=s_2=n(n+1)/2=120$, with QMC points in dimension~$s_1+s_2=240$.
The error was calculated by comparison to the reference value obtained with $N=1024$. The results, shown in Table~\ref{tab: random lambda mu}, again exhibit order $N^{-2}$ ($\approx N^{-\min(1/p,1/q)}= N^{-1/p}= N^{-1/q})$ convergence, which confirms the predicted theoretical convergence results in Theorem \ref{sec: main results}  (or Theorem~\ref{prop:qmc6.3}).


\end{document}